\documentclass{amsart}
\usepackage{graphicx}
\usepackage{eepic}
\usepackage{verbatim}
\usepackage{amsmath, amstext, amssymb, amsthm, amscd}
\usepackage{url}

\numberwithin{equation}{section}
\numberwithin{figure}{section}
\title{Elliptic Fibrations on Supersingular K3 Surface with Artin invariant 1 in characteristic 3}
\author{Tathagata Sengupta}
\address{Department of Mathematics, University of Hyderabad, Hyderabad, India}
\email{tsengupta@gmail.com}
\theoremstyle{definition}
\newtheorem{definition}{Definition} 
\theoremstyle{plain}
\newtheorem{theorem}{Theorem} 
\theoremstyle{plain}
\newtheorem{lemma}{Lemma}
\theoremstyle{plain}

\theoremstyle{remark}
\newtheorem{remark}{Remark} 
\theoremstyle{definition}

\thanks{The author thanks his thesis advisor Prof. Abhinav Kumar for his precious guidance. And also the mathematics departments of Brandeis University, Massachusetts Institute of Technology and University of Hyderabad for invaluable help during the period of this project.}

\date{April 23, 2012}

\begin{document}

\begin{abstract}
We describe elliptic models with section on the Shioda supersingular $K3$ surface $X$ of Artin invariant $1$ over an algebraically closed field of characteristic $3$. We compute elliptic parameters and Weierstrass equations for the fifty two different fibrations, and analyze some of the reducible fibers and Mordell-Weil lattices.
\end{abstract}

\maketitle

\section{Introduction}

\subsection{Definitions and examples}
\begin{definition}
$K3$ \textit{surfaces} are smooth algebraic surfaces $X$ for which the canonical divisor $\mathcal{K}_X \cong \mathcal{O}_X$ and $H^1(X,\mathcal{O}_X) = 0$.
\end{definition}

\begin{definition}
\textit{Supersingular surfaces} (in the sense of Shioda) are those for which all cycles in the $2-$dimensional $\acute{e}$tale cohomology groups are algebraic.
\end{definition}

The \textit{N$\acute{\text{e}}$ron-Severi group}, the group of divisors modulo algebraic equivalence, denoted by $NS(X)$, is a lattice of rank, say $\rho$. The number $\rho$ is called the rank of the $K3$. In characteristic $0$, we have the inequality $\rho\leq b_2-2$, where $b_2$ is the second Betti number. In characteristic $p$ however, we can only say $\rho\leq b_2$. In the case of $K3$ surfaces, $b_2= 22$. Thus, $\rho\leq 22$. According to the definition, (Shioda) supersingular $K3$ surfaces are the ones for which $\rho=22$. The existence of such surfaces is a phenomenon particular to positive characteristic. We now list some of the basic properties of $K3$ surfaces. 
The geometric genus, $p_g= \dim H^2(X,\mathcal{O}_X)=h^{0,2}=1$, because, by Serre duality, $h^2(\mathcal{O}_X)=h^0(X,\mathcal{K}_X \otimes\mathcal{O}_X)= h^0(\mathcal{O}_X))$. Its arithmetic genus, $p_{a}=1$. Hence, the topological Euler characteristic $\chi(\mathcal{O}_X)= h^0(\mathcal{O}_X)-h^1(\mathcal{O}_X)+h^2(\mathcal{O}_X)=2$. We also have that $c_1^2(X)=0$, since $c_1^2(X)=\mathcal{K}_X\cdot \mathcal{K}_X$. Using Noether's formula $c_1^2(X)+c_2(X)=12\chi(\mathcal{O}_X)$, we get the second Chern number, $c_2(X)=24$. Since $b_1(X)=2h^1(\mathcal{O}_X)$, we get that $b_1(X)=0$. Thus, the Betti numbers must be $$b_0=1, b_1=0, b_2=22, b_3=0, b_4=1$$
Some of the best known examples of supersingular $K3$ surfaces are:
\begin{enumerate} 
\item
  \textit{Fermat Quartic}. The Fermat quartic is a hypersurface in $\mathbb{P}^3$ given by the equation $$x^4+y^4+z^4+w^4=0$$
It is a supersingular $K3$ surface over fields of characteristic $3$ mod $4$. 

\item
\textit{Kummer Surfaces}. The Kummer surface of the product of $2$ supersingular elliptic curves over fields of odd characteristic is a supersingular $K3$ surface. We describe in detail the Kummer surface associated to the product of the supersingular elliptic curve $y^2=x^3-x$ with itself, over a field of characteristic $3$.

\item
\textit{Double sextic}. The $double$ $sextic$ is a double cover of the projective plane $\mathbb{P}^2$, branched over a sextic. A $K3$ surface with polarization of degree $2$ can be written as a double cover of $\mathbb{P}^2$, branched along a sextic. The supersingular $K3$ surface of Artin invariant $1$ in characteristic $3$ admits such a model. 
\end{enumerate}

Many of the ideas underlying this work are based on the work of Tetsuji Shioda \cite{Sh1}, \cite{Sh2} Noam Elkies, Matthias Schuett and Abhinav Kumar \cite{Ku1}, \cite{Ku2}, among others. The author is grateful to all the above mathematicians for their amazing contributions and insights. 

\subsection{Elliptic Fibrations}

Let $k$ be an algebraically closed field, and let $C$ be a smooth curve over $k$. 
\begin{definition}
An \textit{elliptic surface} $X$ over $C$ is a smooth projective surface $X$ with an elliptic fibration over $C$, that is, there is a surjective morphism $f:X\to C$ such that all but finitely many fibers are smooth curves of genus $1$, no fiber contains an exceptional curve of the first kind (that is $X\to C$ is \textit{relatively minimal}). Moreover, $f$ has a \textit{section}, that is, a smooth morphism $s:C\to X$ such that $f\circ s=id_C$. We also require that $X$ has at least one singular fiber, so that it is not isomorphic to the product $E\times C'$ for some elliptic curve $E$ after base-change by a finite etale map $C'\to C$. 
\end{definition}
We can use the zero section to describe the Weierstrass form. The generic fiber $E$ can then be regarded as an elliptic curve over the function field $k(C)$. The general Weirstrass form looks like 

$$y^2+a_1(t)xy+a_3(t)y=x^3+a_2(t)x^2+a_4(t)x+a_6(t), \text{ with }a_i(t)\in k(t)$$

which, in characteristic $3$, can be simplified to $$y^2=x^3+a_2(t)x^2+a_4(t)x+a_6(t)$$

 The sections form an abelian group $E(k(C))$. The zero section will be denoted by $O$ throughout this paper. The sections are in a natural one-to-one correspondence with the $k(C)$-rational points on $E$.\\

The \textit{Kodaira-N$\acute{\text{e}}$ron model} gives a way to associate an elliptic surface $X\to C$ over the ground field $k$, given the generic fiber $E$ over $k(C)$. The first step is to remove all the points from $C$ at which the Weierstrass form is singular, and then filling in suitable singular fibers. This process is formally described by Tate's algorithm which we use extensively. It is a well-known fact that the Kodaira-N$\acute{\text{e}}$ron model is unique, given the Weierstrass form.  

\subsubsection{Singular fibers} 
The singular fibers can be read off from the Weierstrass form using Tate's algorithm. All the irreducible components of a singular fiber are rational curves. If a singular fiber is irreducible, then it is a rational curve with a node or a cusp. In Kodaira's classification of singular fibers, these are denoted by I$_1$ and II respectively. If a singular fiber is reducible, then every component is a rational curve with self-intersection $-2$ on the surface. 

\subsubsection{Elliptic K3 surfaces}
In this case, the base curve is $\mathbb{P}^1$. The degrees of the coefficients $a_i$ are restricted by the condition deg$(a_i(t))\leq 2i$, and some deg$(a_i(t))>i$, or else it gives a rational elliptic surface. At $t=\infty$, we change variables by $s=\frac{1}{t}$, and the coefficients become $a_i'(s)=s^{2i}a_i(\frac{1}{s})$. Thus we can homogenise the coefficients to polynomials of degree $2i$ in two variables, $s$ and $t$. The discriminant becomes then a homogeneous polynomial of degree $24$ (in $s$ and $t$). Thus, the Weierstrass form of a elliptic $K3$ surface over $\mathbb{P}^1$ can be seen as a hypersurface in the weighted projective space $\mathbb{P}[1,1,4,6]$. 

\subsubsection{Neron-Severi lattice and Mordell-Weil group}
\begin{definition}
A key invariant of a $K3$ surface $X$ is its N$\acute{\text{e}}$ron-Severi lattice $NS(X)=NS_{\bar{k}}(X)$. This is the usual N$\acute{\text{e}}$ron-Severi group of divisors defined over $\bar{k}$ modulo algebraic equivalence, equipped with the symmetric integer-valued bilinear form induced from the intersection pairing on the divisors. For $K3$ surfaces, the notions of algebraic equivalence, numerical equivalence and linear equivalence coincide, which implies that the \textit{Picard group} of a $K3$ surface is naturally isomorphic to the N$\acute{\text{e}}$ron-Severi group. 
\end{definition}
For a $K3$ surface, this is an abelian group and the pairing is \textit{even}, that is, $C\cdot C \in 2\mathbb{Z}$ for all $C\in NS(X)$. By the Hodge index theorem, the pairing is non-degenerate of signature $(1,\rho-1)$, where $\rho$ is the rank of $NS(X)$ (also known as the $rank$ of the $K3$). In our case, $\rho=22$, in other words, $X$ is a \textit{supersingular} $K3$ surface. By Artin \cite{Ar}, the discriminant of the Neron-Severi lattice is $-p^{2\sigma}$, where $p$ is the characteristic of the base field $k$, and $\sigma$ is an invariant for the surface, known as the \textit{Artin invariant}. It is also known that $1\leq \sigma \leq 10$. In this paper, we will largely deal with the supersingular $K3$ surface over characteristic $3$, with Artin invariant $1$. Note the definite article here, since there can be only one, upto isomorphism, as shown in \cite{Og}. \\

If $X\to C$ is an elliptic fibration with a zero section, $NS(X)$ contains two special classes, the fiber $F$ (preimage of any point of $\mathbb{P}^1$ under $f$), and the image of the zero section $O$. The intersection pairing they satisfy is $F\cdot F=0$, $O\cdot O =-2$ and $F\cdot O=1$. Hence the sublattice $\mathcal{U}$ generated by $F$ and $O$ is isomorphic to the \textit{hyperbolic plane} ($\mathcal{U}$). Conversely, any copy of $\mathcal{U}$ in $NS(X)$ describes $X$ as an elliptic surface. One of the generators or its negative is \textit{effective}, and has two independent sections, whose ratio gives the map to $\mathbb{P}^1$.  

\begin{definition}
The \textit{essential lattice} is the orthogonal complement of the copy of $\mathcal{U}$ in $NS(X)$, and is denoted by $NS_{\text{ess}}$. It is a positive definite lattice. 
\end{definition}

\begin{definition}
The \textit{Mordell-Weil group} of the surface $f: X\to C$ is the (abelian) group of sections from $C\to X$. This is also naturally identified with the $k(C)$-rational points of $E$, denoted by $E(k(C))$, where $E$ is the generic fiber for the elliptic fibration defined by $f: X\to C$.
\end{definition}
According to teh Mordell-Weil theorem \cite{Si}, $E(k(C))$ is a finitely generated group. 

\begin{definition}
The \textit{trivial lattice} $T$ is the lattice generated by the classes $O$, $F$ and the $F_{\nu,i}$, where $O$ is the zero section, $F$ is the class of the fiber, and $F_{\nu,i}$ is the $i$-th non-identity component of the reducible fiber at $\nu$. Here, the \textit{identity component} of a reducible fiber is the component intersecting the zero section. 
\end{definition}
Thus, if $m_\nu$ is the number of components of the fiber $F_\nu$, then the rank$(T)=2+\sum_{\nu\in P}(m_\nu-1)$, where $P$ is the set of points of $\mathbb{P}^1(k)$ where the fibers are reducible. The Shioda-Tate formula gives the following relation between the Mordell-Weil group and the Neron-Severi lattice: 

$$E(k(C))\cong NS(X)/T$$

Let $R \subset  NS_{\text{ess}}$ be the \textit{root lattice} of  $NS_{\text{ess}}$, that is the sublattice spanned by the vectors of norm $2$ (known as \textit{roots}). Then  $NS_{\text{ess}}/R$ can be canonically identified with the Mordell-Weil group $E(k(C))$. $R$ itself is a direct sum of root lattices $A_n$ $(n\geq 1)$, $D_n$ $n\geq 4$, $E_6$, $E_7$ or $E_8$, with each factor indicating a reducible fiber of the corresponding type. 

\subsubsection{Dynkin diagrams} 
The Dynkin diagrams corresponding to the root systems are given in Figure $1.1$. Here the nodes indicate roots (rational curves with self-intersection $-2$), and two nodes are connected by an edge if and only if the corresponding rational curves intersect, in which case the intersection number is $1$.

\begin{figure}[ht!]
\setlength{\unitlength}{.45in}
\begin{picture}(10,4.5)(0.4,-0.3)
\thicklines

% D4

  \put(1,3){\circle*{.1}}
    \put(1,3.55){\makebox(0,0)[l]{${A}_1$}}

% E7
\put(4.5,3){\line(1,0){0.5}}
\multiput(4.5,3)(0.5,0){2}{\circle*{.1}}
    \put(5,3.55){\makebox(0,0)[l]{${A}_2$}}

% E8
\put(8,3){\line(1,0){0.25}}
\multiput(8,3)(1.5,0){2}{\circle*{.1}}

  \put(9.25,3){\line(1,0){0.25}}
%    \put(10,1.5){\circle*{.1}}
    \put(9,3.55){\makebox(0,0)[l]{${A}_n\;(n>2)$}}
      \put(8.5,3){\makebox(0,0)[l]{$\hdots$}}

% D4
\put(0.5,1.5){\line(1,0){1}}
\multiput(0.5,1.5)(1,0){2}{\circle*{.1}}
\put(1,1.5){\line(0,1){0.5}}
  \put(1,2){\circle*{.1}}
  \put(1,1.5){\circle*{.1}}
    \put(1.5,2.05){\makebox(0,0)[l]{${D}_4$}}

% E7
\put(4,1.5){\line(1,0){1.5}}
\multiput(4.5,1.5)(1,0){2}{\circle*{.1}}
\put(4.5,1.5){\line(0,1){0.5}}
  \put(4.5,2){\circle*{.1}}
    \put(4,1.5){\circle*{.1}}
  \put(5,1.5){\circle*{.1}}
    \put(5.5,2.05){\makebox(0,0)[l]{${D}_5$}}

% E8
\put(7.5,1.5){\line(1,0){0.75}}
\multiput(7.5,1.5)(0.5,0){2}{\circle*{.1}}
\put(8,1.5){\line(0,1){0.5}}
  \put(8,2){\circle*{.1}}
  \put(9.5,1.5){\circle*{.1}}
  \put(9.25,1.5){\line(1,0){0.75}}
    \put(10,1.5){\circle*{.1}}
    \put(9,2.05){\makebox(0,0)[l]{${D}_n\;(n>5)$}}
      \put(8.5,1.5){\makebox(0,0)[l]{$\hdots$}}

% E6
\put(0,0){\line(1,0){2}}
\multiput(0,0)(0.5,0){5}{\circle*{.1}}
\put(1,0){\line(0,1){0.5}}
  \put(1,0.5){\circle*{.1}}
    \put(1.5,0.55){\makebox(0,0)[l]{${E}_6$}}

% E7
\put(3.5,0){\line(1,0){2.5}}
\multiput(3.5,0)(0.5,0){5}{\circle*{.1}}
\put(5,0){\line(0,1){0.5}}
  \put(5,0.5){\circle*{.1}}
  \put(6,0){\circle*{.1}}
    \put(5.5,0.55){\makebox(0,0)[l]{${E}_7$}}

% E8
\put(7.5,0){\line(1,0){3}}
\multiput(7.5,0)(0.5,0){6}{\circle*{.1}}
\put(8.5,0){\line(0,1){0.5}}
  \put(8.5,0.5){\circle*{.1}}
  \put(10.5,0){\circle*{.1}}
    \put(10,0.55){\makebox(0,0)[l]{${E}_8$}}

\end{picture}
\caption{Dynkin diagrams}
\label{Fig:Dynkin}
\end{figure}
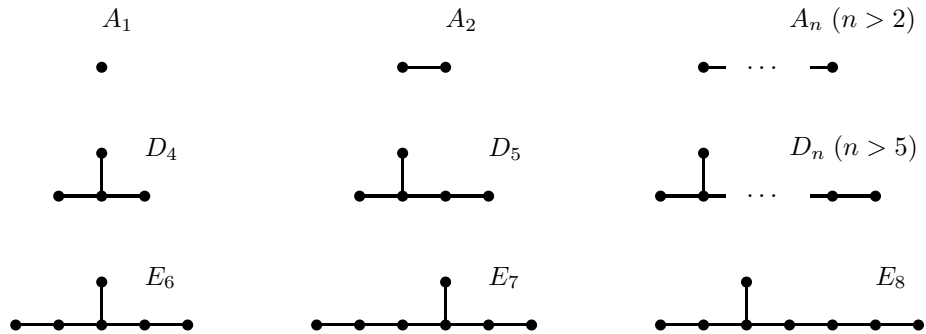

The determinants of the negative-definite lattices defined by the respective Dynkin diagrams are:\\ 
$\text{det}(A_n)=(-1)^n(n+1)$, det$(D_n)=(-1)^n4$ $(n\geq 4)$\\
det$(E_6)=3$, det$(E_7)=-2$, det$(E_8)=1$

\subsubsection{Height Pairing}
The \textit{height pairing formula} between two sections $P$ and $Q$ is given by 
$$\langle P,Q \rangle=\chi(X)+P\cdot O+Q\cdot O-P\cdot Q -\sum_\nu\text{contr}_\nu(P,Q)$$
where $\chi(X)$ is the Euler characteristic of the surface ($=2$ in the case of $K3$ surfaces), and the correction terms contr$_\nu(P,Q)$ depend on the components of the fibers that are met by $P$ and $Q$. More specifically, let $A_\nu$ be the intersection matrix of the non-identity components of the fibers $F_\nu$, $\nu\in \mathbb{P}^1$ 
$$A_\nu=(\Theta_{\nu,i},\Theta_{\nu,j})_{1\leq i,j\leq m_{\nu}-1}$$
Then if $P$ meets $\Theta_{\nu,i}$, and $Q$ meets $\Theta_{\nu,j}$, then the local contribution term is given by 

$$\text{contr}_\nu(P,Q)= 0 \text{ if } ij=0,\text{ or }= -(A_\nu^{-1})_{i,j}\text{ if }ij\neq 0$$

We specialize $P=Q$ to get the $height$ of a single section $P$. Thus, 

$$h(P)=\langle P,P\rangle= 2\chi(X)+2P\cdot O -\sum_\nu\text{contr}_\nu(P,P)$$

The contribution terms are given in the following table. The conventions followed are the following: 
\begin{enumerate}
\item The components of the $A_n$ fibers are numbered cyclically by $\Theta_0$, $\Theta_1$,$\cdots$,$\Theta_n$, where $\Theta_0$ is the identity component. 
\item At additive fibers, only the simple components are named. 
\item The simple components of the $D_n$ fibers are numbered as $\Theta_1$ for the near component, and $\Theta_2$, $\Theta_3$ for the far components (with respect to the identity component).
\end{enumerate}

\begin{center}
\begin{tabular}{ccccl}\hline
Dynkin Diagram & $A_{n-1}$ & $E_6$ & $E_7$ & $D_{n+4}$\\
\hline \hline
$i=j$ & $\frac{i(n-i)}{n}$ & $\frac{4}{3}$ & $\frac{3}{2}$ & $1$, $i=1$\\
& & & & $1+\frac{n}{4}$, $i=2,3$\\
\hline
$i<j$ & $\frac{i(n-j)}{n}$ & $\frac{2}{3}$ & - & $\frac{1}{2}$, $i=1$\\
& & & & $\frac{1}{2}+\frac{n}{4}$, $i=2$\\
\hline
\end{tabular}
\end{center}

As described in \cite{Sh2}, we compute the determinant of the Neron-Severi using
$$\text{disc}(NS(X))=\frac{(-1)^{\text{rank}(E(k(C)))}\text{disc}(T(X))\text{disc}(MWL(X))}{|E(k(C))_{\text{tor}}|^2}$$
where the discriminant of the Mordell-Weil Lattice (the free part generated by non-torsion sections) is calculated using the height pairing between the sections. $(E(k(C)))_{\text{tor}}$ is the group of torsion sections. 
\clearpage

\section{Elliptic Fibrations for the Supersingular K3 Surface with $\sigma=1$ over char $3$}
\subsection{Introduction}
Our goal is to use  $2$-neighbour and $3$-neighbour constructions in root lattices to obtain all possible elliptic fibrations of the supersingular surface of Artin invariant $1$ over characteristic $3$. Many of the ideas underlying the work in this paper are based on the works of Tetsuji Shioda \cite{Sh1}, \cite{Sh2} Noam Elkies, Matthias Schuett and Abhinav Kumar \cite{Ku1}, \cite{Ku2}, among others. From the works of Rudakov-Shafarevich and Ichiro Shimada we know descriptions of the abstract N$\acute{\text{e}}$ron-Severi lattices for other characteristics and Artin invariants. The methods used in this paper work more generally for those cases as well, although the computations can become more cumbersome, albeit finite. Ideally one should be able to design a computer program which could compute these fibrations in a reasonable time. \\
A few definitions: 
\begin{definition}
A $Niemeier$ $lattice$ is the orthogonal complement of a copy of the hyperbolic plane $\mathcal{U}$ in the even unimodular lattice of signature $(1,25)$, II$_{(1,25)}$. There are $24$ such lattices, which have been enumerated by Hans-Volker Niemeier, and except one, all of them have root sublattices. These are all the unimodular even positive definite lattices of rank $24$. 

\end{definition}

\begin{definition}
A \textit{root} in a lattice is an element of norm $-2$.
\end{definition}

\begin{definition}
A \textit{root lattice} is one that is generated by its set of roots.
\end{definition}

The only Niemeier lattice that does not have any roots is the famous \textit{Leech lattice}. All the other Niemeiers have finite index root sublattices. The entire list of root sublattices of the Niemeier lattices is as follows 
$$A_1^{24}, A_2^{12}, A_3^8, A_4^6, A_5^4D_4, D_4^6, A_6^4, A_7^2D_5^2,$$ 
$$A_8^3, A_9^2D_6, D_6^4, E_6^4, A_{11}D_7E_6, A_{12}^2, D_8^3, A_{15}D_9,$$ 
$$A_{17}E_7, D_{10}E_7^2, D_{12}^2, A_{24}, D_{16}E_8, E_8^3, D_{24}$$

Here we prove that there is a finite list of possible elliptic fibrations on the supersingular $K3$ over a field of characteristic $3$ which we assume to be algebraically closed. We describe the singular fiber, using the Neimeier lattice description. We go on to describe the fibrations, and give Weierstrass equations for these models.

\begin{lemma}
There is a unique embedding of $A_2$ into each of the root lattices $R$, upto the action of the automorphism group
\end{lemma}
\begin{proof}
From the general theory of Weyl group actions on root lattices, we know that the action is simply transitive. Thus, if $D$ and $D'$ are two irreducible Dynkin diagrams giving a basis for the corresponding root lattice (as described in one of the earlier sections), then there is an element $\sigma\in W(R)$ such that $\sigma(D)=D'$. Based on this, we want to prove that given any copy of $A_2$ within a root lattice, it can be extended to the complete Dynkin diagram, such that the copy of $A_2$ forms 
\begin{enumerate}
\item the first $2$ roots of the $A_n$ diagram, in case of an $A_n$ root lattice,
\item the first $2$ roots of the long tail of the $D_n$ $(n>4)$ diagram, in case of the $D_n$ root lattice, (note that any embedding of $A_2$ is equivalent in a $D_4$)
\item the $2$ roots of the short tail of the $E_6$ diagram, in case of the $E_6$ lattice,
\item the $2$ roots of the tail of length $2$ (the other $2$ tails being of length $1$ and $3$, all attached to the one single node of valency $3$), in case of the $E_7$ lattice, and 
\item the first $2$ roots of the longest tail of the $E_8$ diagram, in case of the $E_8$ lattice.
\end{enumerate}
In the case of $A_n$, let us denote the nodes by $[a_1,a_2,a_3,a_4,\cdots a_n]$. If our $A_2$ maps to $\{a_2,a_3\}$, we can extend the chain of nodes given by $[a_2,a_3,a_4,\cdots a_n]$ by the root $-(a_1+a_2+a_3+\cdots a_n)$. This clearly gives us a new $A_n$ diagram, with $\{a_2,a_3\}$ as the first $2$ roots. If $A_2$ maps to $\{a_3,a_4\}$, we can extend the chain of nodes given by $[a_3,a_4,a_5,\cdots a_n]$ by the roots $\{-(a_2+a_3+\cdots a_n),-a_1\}$. This gives us a new $A_n$ diagram, with $\{a_3,a_4\}$ as the first $2$ roots. If $A_2$ maps to $\{a_4,a_5\}$, we can extend the chain of nodes given by $[a_4,a_5,\cdots a_n]$ by the roots $\{-(a_3+\cdots a_n),-a_2,-a_1\}$. This gives us a new $A_n$ diagram, with $\{a_3,a_4\}$ as the first $2$ roots. We can do this with any embedding of $A_2$ into $A_n$. \\
In the case of $D_n$, given by $[a_1,a_2,a_3,\cdots,a_{n-1},a_n]$, where $a_n$ is connected to $a_{n-2}$ and no other root. If $A_2$ maps to $\{a_2,a_3\}$, we can extend the straight chain given by $[a_2,a_3,\cdots,a_n]$ to a $D_n$, using $-(a_1+a_2+\cdots+a_{n-2}+a_n)$ and $-(a_1+a_2+\cdots+a_{n-2}+a_{n-1})$. The other cases are very similar to this. \\
In the case of $E_6$, let us denote the diagram by $[a_1,a_2,a_3,a_4,a_5,a_6]$, where $[a_1,a_2,a_3,a_4,a_5]$ gives an $A_5$, and $a_6$ is connected to $a_3$. If our $A_2$ maps to $\{a_1,a_2\}$, we can extend the chain $[a_1,a_2,a_3,a_6]$ by the roots $-(a_1+2a_2+2a_3+a_4+a_6)$ and $-a_5$, which gives an $E_6$ with the $A_2$ now mapping to the shortest tail of an $E_6$ diagram.\\
We can do similar extensions to other embeddings of $A_2$, and also similarly in the case of the other $E_n$'s. 
\end{proof}

\begin{theorem}
There are $52$ possible elliptic fibrations on the supersingular $K3$ surface of Artin invariant $1$ over an algebraically closed field of characteristic $3$
\end{theorem}

\begin{remark}
Although we assume the ground field to be algebraically closed, in our calculations we have to go only upto $\mathbb{F}_9$. 
\end{remark}

\begin{proof}
Let $X/k$ be the supersingular $K3$ of Artin invariant $1$, $k$ is algebraically closed field of characteristic $3$. The N$\acute{\text{e}}$ron-Severi lattice of $X$, $NS(X) (=N)$ is an invariant of $X$. From the classification of even non-degenerate lattices, we know that $NS(X)\cong \mathcal{U}\oplus A_2^2\oplus E_8^2$, where $\mathcal{U}$ is the hyperbolic plane. This is a $22$ dimensional negative definite lattice of signature $(1,21)$ and discriminant $-3^2$. Using well-known results from lattice theory, we know that $N$ can be embedded into the even unimodular lattice of signature $(1,25)$, II$_{(1,25)}$. Now, as we described before, the choice of the class of a fiber and a zero section on $X$ determines a copy of the hyperbolic plane $\mathcal{U}$ in $NS(X)$, and vice-versa. Describing embeddings $\mathcal{U}\hookrightarrow NS(X)$ is equivalent to describing the embeddings $NS(X)^{\perp}\hookrightarrow \mathcal{U}^\perp$, where the orthogonal complements are taken in the ambient $II_{(1,25)}$. We know that the $\mathcal{U}^\perp$'s are by definition the Niemeier lattices. Now, according to a result due to Nikulin, let $M\hookrightarrow L$ be a primitive embedding of non-degenerate even lattices, and suppose that $L$ is unimodular. Then the discriminant forms $q_{M^\perp}\cong -q_M$. Conversely if $M_1$ and $M_2$ are non-degenrate even lattices which satisfy $q_{M_1}\cong -q_{M_2}$, then there is a primitive embedding of $M_1\hookrightarrow L$ such that $M_1^\perp\cong M_2$. In our case, let $M=\mathcal{U}\oplus A_2^2\oplus E_8^2$. Thus, the discriminant of $M$ is $-9$. Using the above result from Nikulin \cite{N1},\cite{N2} the discriminant of $M^\perp$ is $9$, and it is of rank $4$. Thus, $M^\perp\cong A_2^2$, using to the classification of lattices. We thus need to figure out the number of embeddings of $A_2^2$ in each of the Niemeier lattices, upto isomorphism. \\
Now, using the previous lemma, we can describe the orthogonal complement of $A_2$ in each of the root systems, and successively for the second copy of $A_2$. The uniqueness of these embeddings follows from the preceding lemma. The list of the orthogonal complements are given in the table below. 
\end{proof}
The following table gives the orthogonal complement of $A_2^2$ in each of the different Niemeier lattices. This calculation is due to Noam Elkies and Matthias Schuett.
\begin{center}
\begin{tabular}{|c|c|c|}\hline
Root System $D$ & $A_2^\perp \subset D$ & $(A_2^2)^\perp \subset D$\\
\hline\hline
$A_2$ & $0$ & $0$\\
\hline
$A_3$ & $0$ & $0$\\
\hline
$A_n$  for $n>3$ & $A_{n-3}$ & $A_{n-6}$ for $n>6$ ($0$ for $A_5,A_6$)\\
\hline
$D_4$ & $0$ & $0$\\
\hline
$D_5$ & $A_1^2$ & $0$\\
\hline
$D_6$ & $A_3$ & $0$\\
\hline
$D_n$ for $n>6$ & $D_{n-3}$ & $D_{n-6}$ for $n>9$ ($0$ for $D_7$, $A_1^2$ for $D_8$, $A_3$ for $D_9$)\\
\hline
$E_6$ & $A_2^2$ & $A_2$\\
\hline
$E_7$ & $A_5$ & $A_2$\\
\hline
$E_8$ & $E_6$ & $A_2^2$\\
\hline
\end{tabular}
\end{center}
Thus the lattices that can contain $A_2^2$ are $A_n (n>4), D_n (n>5)$, and all of the $E_n$'s.  \\
Let $N$ be a Niemeier lattice, and $N'$ the orthogonal complement of $A_2^2$ in $N$. Following is the list of $\text{Roots}(N)$ and $\text{Roots}(N')$. Borrowing notation from Noam Elkies, $m:$ stands for the lattice obtained by extracting $A_2^2$ from the $m$-th factor of $\text{Roots}(N)$, and $mn:$ for the lattice obtained by extracting $A_2$'s from the $m$-th and $n$-th factor.  

\begin{center}
\begin{tabular}{|c|c|c||c|c|c|}\hline
$\text{Roots}(N)$ & $\text{Roots}(N')$ & rk$(MW)$ & $\text{Roots}(N)$ & $\text{Roots}(N')$ & rk$(MW)$\\
\hline \hline
   $A_2^{12}$  &    $ A_2^{10}$&             $0$ & $A_{12}^2$    &  $1: A_6 A_{12}$          & $2$\\
\hline  
   $A_3^8$      &   $A_3^6$     &            $2$ &   &  $11: A_9^2$              & $2$\\

\hline
   $A_4^6$       &  $A_1^2 A_4^4$&           $2$ &  $D_8^3$       & $1: A_1^2 D_8^2$          & $2$\\
\hline
   $D_4^6$        & $D_4^4$       &          $4$ &     & $11: D_5^2 D_8$           & $2$\\
\hline
   $D_4A_5^4$     & $12: A_2 A_5^3$ &        $3$ & $D_9 A_{15}$  & $1: A_3 A_{15}$  & $2$\\
\hline      
            & $2: D_4 A_5^3$   &       $1$ &          & $2: A_9 D_9$               &$2$\\
\hline      
            & $22: A_2^2 D_4 A_5^2$&   $2$&
           & $12: D_6 A_{12}$           & $2$\\
\hline
   $A_6^4$       &  $1: A_6^3$            &  $2$&$E_7 A_{17}$     & $1: A_2 A_{17}$           & $1$\\
\hline      
          &  $11: A_3^2 A_6^2$       & $2$&
           & $2: E_7 A_{11}$            &$2$\\
\hline  
   $D_5^2 A_7^2$  & $11: A_1^4 A_7^2$       & $2$&  &  $12: A_5 A_{14}$           & $1$\\
\hline      
           &  $2: A_1 D_5^2 A_7$       &$2$&$E_7^2 D_{10}$    & $1: A_2 E_7 D_{10}$        & $1$\\
\hline      
           & $12: A_1^2 A_4 D_5 A_7$   &$2$& & $11: A_5^2 D_{10}$          &$0$\\
\hline          
            &$22: A_4^2 D_5^2$         &$2$& & $2: D_4 E_7^2$              &$2$\\
\hline
   $A_8^3$        & $1: A_2 A_8^2$        &   $2$& & $12: A_5 D_7 E_7$           &$1$\\
\hline    
           & $11: A_5^2 A_8$        &   $2$& $D_{12}^2$     &  $1: D_6 D_{12}$         &    $2$\\
\hline 
   $D_6^4$       & $1: D_6^3$              &  $2$&  & $11: D_9^2$              &   $2$\\
\hline     
           & $11: A_3^2 D_6^2$       &  $2$&        
  $E_8^3$         &$1: A_2^2 E_8^2$         &    $0$\\
\hline
   $D_6 A_9^2$   &  $1: A_9^2$              & $2$&
            &$11: E_6^2 E_8$          &    $0$\\
\hline     
           & $2: A_3 D_6 A_9$         & $2$&$ E_8 D_{16}$    &  $1: A_2^2 D_{16}$        &    $0$\\
\hline      
           & $12: A_3 A_6 A_9$        & $2$& &   $2: E_8 D_{10}$          &   $2$\\
\hline     
           & $22: A_6^2 D_6$          & $2$& &    $12: E_6 D_{13}$         &   $1$\\
\hline
   $E_6^4$       & $1: A_2 E_6^3$           & $0$& $A_{24}$   &       $A_{18}$                &  $2$\\
\hline     
           & $11: A_2^4 E_6^2$        & $0$& $D_{24}$  &        $D_{18}$                &  $2$\\
\hline
 $E_6 D_7 A_{11}$ & $1: A_2 D_7 A_{11}$      & $0$&&&\\
\hline      
           &  $2: E_6 A_{11}$          & $3$ &&&\\
\hline      
           & $12: A_2^2 D_4 A_{11}$    & $1$&&&\\
\hline      
           & $3: E_6 D_7 A_8$          &$2$&&&\\
\hline      
           & $13: A_2^2 D_7 A_8$       &$1$&&&\\
\hline      
           & $23: D_4 E_6 A_8$         &$2$&&&\\
\hline
       
\end{tabular}
\end{center}
Now we begin writing down the elliptic models. 

\section{Fibration 1 : $A_{11},A_2, D_7$}
Weierstrass equation: $y^2=x^3+2(t^3+1)x^2+t^6x$\\ 
This fibration corresponds to the fiber type of $A_{11}A_2D_7$. In the diagram below, we have the $A_{11}$ fiber at $t=0$, the $A_2$ fiber at $t=1$, and the $D_7$ fiber at $t=\infty$, given by the nodes $a_i$'s, $b_j$'s  and $c_k$'s respectively. The zero section is given by the node $O$, and the torsion section by $T$. These nodes represent the roots of $NS(X)$, which correspond smooth rational curves of self-intersection $-2$. 

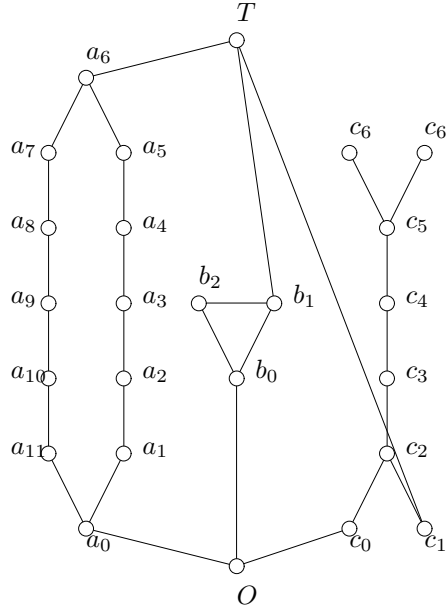
\begin{center}
\setlength{\unitlength}{1cm}
\begin{picture}(10.0,8)(0.0,0.0)
\filltype{white}

\path(1.0,1.0)(1.5,2.0)
\path(1.0,1.0)(0.5,2.0)
\path(1.5,2.0)(1.5,3.0)
\path(1.5,3.0)(1.5,4.0)
\path(1.5,4.0)(1.5,5.0)
\path(1.5,5.0)(1.5,6.0)
\path(1.5,6.0)(1.0,7.0)
\path(0.5,2.0)(0.5,3.0)
\path(0.5,3.0)(0.5,4.0)
\path(0.5,4.0)(0.5,5.0)
\path(0.5,5.0)(0.5,6.0)
\path(0.5,6.0)(1.0,7.0)

\path(3.0,3.0)(3.5,4.0)
\path(3.0,3.0)(2.5,4.0)
\path(2.5,4.0)(3.5,4.0)

\path(4.5,1.0)(5.0,2.0)
\path(5.5,1.0)(5.0,2.0)
\path(5.0,2.0)(5.0,3.0)
\path(5.0,3.0)(5.0,4.0)
\path(5.0,4.0)(5.0,5.0)
\path(5.0,5.0)(4.5,6.0)
\path(5.0,5.0)(5.5,6.0)

\path(3.0,7.5)(1.0,7.0)
\path(3.0,7.5)(3.5,4.0)
\path(3.0,7.5)(5.5,1.0)

\path(3.0,0.5)(1.0,1.0)
\path(3.0,0.5)(3.0,3.0)
\path(3.0,0.5)(4.5,1.0)

\put(1.0,1.0){\circle*{0.2}}
\put(1.5,2.0){\circle*{0.2}}
\put(1.5,3.0){\circle*{0.2}}
\put(1.5,4.0){\circle*{0.2}}
\put(1.5,5.0){\circle*{0.2}}
\put(1.5,6.0){\circle*{0.2}}
\put(1.0,7.0){\circle*{0.2}}
\put(0.5,2.0){\circle*{0.2}}
\put(0.5,3.0){\circle*{0.2}}
\put(0.5,4.0){\circle*{0.2}}
\put(0.5,5.0){\circle*{0.2}}
\put(0.5,6.0){\circle*{0.2}}

\put(3.0,3.0){\circle*{0.2}}
\put(3.5,4.0){\circle*{0.2}}
\put(2.5,4.0){\circle*{0.2}}

\put(4.5,1.0){\circle*{0.2}}
\put(5.5,1.0){\circle*{0.2}}
\put(5.0,2.0){\circle*{0.2}}
\put(5.0,3.0){\circle*{0.2}}
\put(5.0,4.0){\circle*{0.2}}
\put(5.0,5.0){\circle*{0.2}}
\put(4.5,6.0){\circle*{0.2}}
\put(5.5,6.0){\circle*{0.2}}

\put(3.0,0.5){\circle*{0.2}}
\put(3.0,7.5){\circle*{0.2}}

\put(1.0,0.75){$a_0$}
\put(1.75,2.0){$a_1$}
\put(1.75,3.0){$a_2$}
\put(1.75,4.0){$a_3$}
\put(1.75,5.0){$a_4$}
\put(1.75,6.0){$a_5$}
\put(1.0,7.25){$a_6$}
\put(0.0,2.0){$a_{11}$}
\put(0.0,3.0){$a_{10}$}
\put(0.0,4.0){$a_9$}
\put(0.0,5.0){$a_8$}
\put(0.0,6.0){$a_7$}

\put(3.25,3.0){$b_0$}
\put(3.75,4.0){$b_1$}
\put(2.5,4.25){$b_2$}

\put(4.5,0.75){$c_0$}
\put(5.5,0.75){$c_1$}
\put(5.25,2.0){$c_2$}
\put(5.25,3.0){$c_3$}
\put(5.25,4.0){$c_4$}
\put(5.25,5.0){$c_5$}
\put(4.5,6.25){$c_6$}
\put(5.5,6.25){$c_6$}

\put(3.0,0.0){$O$}
\put(3.0,7.75){$T$}
\end{picture}
\end{center}
The trivial lattice (that is, the sub-lattice of $NS(X)$ spanned by the classes of the zero section and the irreducible components of the fibers) has rank $2+11+2+7=22$. Thus the $MW$-rank is $0$. We have a $4$-torsion section, given by $(t^3,0)$. The discriminant of the trivial lattice is $12\cdot3\cdot 4=144$. Thus, the trivial lattice together with the $4$-torsion section has signature $(1,21)$ and absolute discriminant $3^2$, and is thus equal to the full $NS(X)$. Below are the positions of the reducible fibers, and some of the sections.

\begin{center}
\begin{tabular}{|l|c|}\hline
Position & Kodaira-N$\acute{\text{e}}$ron type\\
\hline \hline
$t=0$ & $A_{11}$\\
\hline
$t=1$ & $A_2$\\
\hline
$t=\infty$ & $D_7$\\
\hline
\end{tabular}
\end{center} 

\begin{center}
\begin{tabular}{|l|c|}\hline
section type & equation\\
\hline \hline
$4$-torsion & $(t^3,0)$\\
\hline
\end{tabular}
\end{center}

\section{Fibration 2 : $D_4,D_4,D_4,D_4$}
Weierstrass equation: $y^2=x^3-t^2(t-1)^2(t+1)^2x$\\
This is the Kummer surface associated to the product of the supersingular elliptic curve $y^2=x^3-x$ with itself, over characteristic $3$. This fibration corresponds to the fiber type of $D_4^4$, with $MW$-rank $4$, and full $2$-torsion sections. In the diagram below, we see the $D_4$ fibers at $t=0$,$1$,$-1$ and $\infty$,given by the nodes $R_{ij}$, $(1\leq i\leq 4, 0\leq j\leq 4)$, the zero section $O$, and one of the non-torsion sections $U$. The nodes represent the roots in the N$\acute{\text{e}}$ron-Severi lattice, which correspond to smooth rational curves of self-intersection $-2$. 

%\begin{figure}[h]
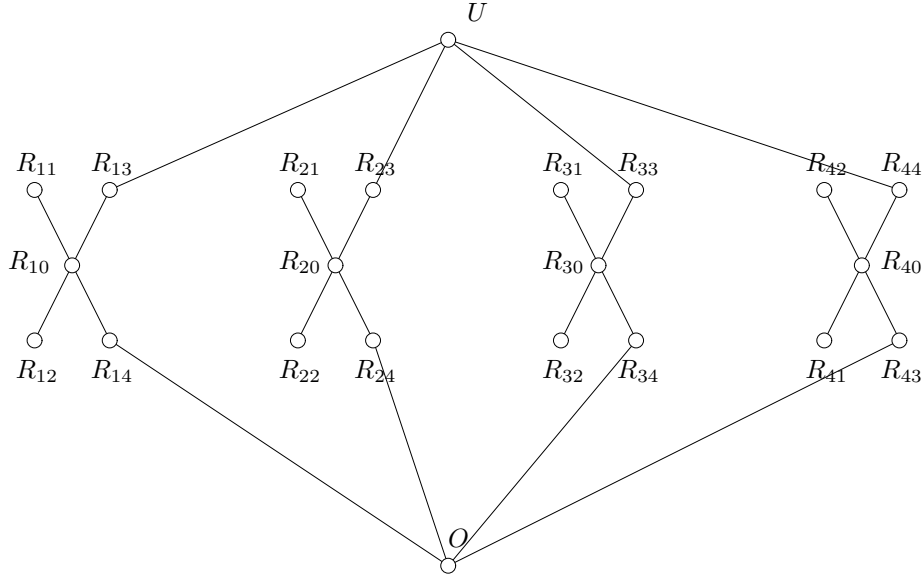
\begin{center}
\setlength{\unitlength}{1cm}
\begin{picture}(12.5,8)(0.0,0.0)
\filltype{white}

\path(1.0,4.0)(0.5,3.0)
\path(1.0,4.0)(0.5,5.0)
\path(1.0,4.0)(1.5,3.0) 
\path(1.0,4.0)(1.5,5.0)

\path(4.5,4.0)(4.0,3.0)
\path(4.5,4.0)(4.0,5.0)
\path(4.5,4.0)(5.0,3.0) 
\path(4.5,4.0)(5.0,5.0)

\path(8.0,4.0)(7.5,3.0)
\path(8.0,4.0)(7.5,5.0)
\path(8.0,4.0)(8.5,3.0) 
\path(8.0,4.0)(8.5,5.0)

\path(11.5,4.0)(11.0,3.0)
\path(11.5,4.0)(11.0,5.0)
\path(11.5,4.0)(12.0,3.0)
\path(11.5,4.0)(12.0,5.0)

\path(6.0,0.0)(1.5,3.0)
\path(6.0,0.0)(5.0,3.0)
\path(6.0,0.0)(8.5,3.0)
\path(6.0,0.0)(12.0,3.0)

\path(6.0,7.0)(1.5,5.0)
\path(6.0,7.0)(5.0,5.0)
\path(6.0,7.0)(8.5,5.0)
\path(6.0,7.0)(12.0,5.0)

\put(1.0,4.0){\circle*{0.2}}
\put(0.5,3.0){\circle*{0.2}}
\put(1.5,3.0){\circle*{0.2}}
\put(0.5,5.0){\circle*{0.2}}
\put(1.5,5.0){\circle*{0.2}}

\put(4.5,4.0){\circle*{0.2}}
\put(4.0,3.0){\circle*{0.2}}
\put(5.0,3.0){\circle*{0.2}}
\put(4.0,5.0){\circle*{0.2}}
\put(5.0,5.0){\circle*{0.2}}

\put(8.0,4.0){\circle*{0.2}}
\put(7.5,3.0){\circle*{0.2}}
\put(8.5,3.0){\circle*{0.2}}
\put(7.5,5.0){\circle*{0.2}}
\put(8.5,5.0){\circle*{0.2}}

\put(11.5,4.0){\circle*{0.2}}
\put(11.0,3.0){\circle*{0.2}}
\put(12.0,3.0){\circle*{0.2}}
\put(11.0,5.0){\circle*{0.2}}
\put(12.0,5.0){\circle*{0.2}}

\put(1.25,2.5){$R_{14}$}
\put(1.25,5.25){$R_{13}$}
\put(0.25,2.5){$R_{12}$}
\put(0.25,5.25){$R_{11}$}
\put(0.15,3.95){$R_{10}$}

\put(4.75,2.5){$R_{24}$}
\put(4.75,5.25){$R_{23}$}
\put(3.75,2.5){$R_{22}$}
\put(3.75,5.25){$R_{21}$}
\put(3.75,3.95){$R_{20}$}

\put(8.25,2.5){$R_{34}$}
\put(8.25,5.25){$R_{33}$}
\put(7.25,2.5){$R_{32}$}
\put(7.25,5.25){$R_{31}$}
\put(7.25,3.95){$R_{30}$}

\put(10.75,2.5){$R_{41}$}
\put(10.75,5.25){$R_{42}$}
\put(11.75,2.5){$R_{43}$}
\put(11.75,5.25){$R_{44}$}
\put(11.75,3.95){$R_{40}$}

\put(6.0,0.0){\circle*{0.2}}
\put(6.0,7.0){\circle*{0.2}}
\put(6.0,0.25){$O$}
\put(6.25,7.25){$U$}

\end{picture}
\end{center}
The trivial lattice (that is, the sub-lattice of $NS(X)$ spanned by the classes of the zero section and the irreducible components of the fibers) has rank $2+4\cdot 4=18$, and discriminant $-16$. The heights of the $4$ non-torsion sections are $2$,$1$,$1$ and $2$. We see that the sublattice generated by the trivial lattice, the torsion sections and the non-torsion sections has signature $(1,21)$ and absolute discriminant $3^2$. Consequently it must be all of $NS(X)$, which we know to be of the same rank and discriminant. 

\begin{center}
\begin{tabular}{|l|c|}\hline
Position & Kodaira-N$\acute{\text{e}}$ron type\\
\hline \hline
$t=0$ & $D_4$\\
\hline
$t=1$ & $D_4$\\
\hline
$t=-1$ & $D_4$\\
\hline
$t=\infty$ & $D_4$\\
\hline
\end{tabular}
\end{center} 

\begin{center}
\begin{tabular}{|l|c|}\hline
section type & equation\\
\hline \hline
$2$-torsion & $(0,0)$\\
\hline
& $(t(t-1)(t+1),0)$\\
\hline
& $(-t(t-1)(t+1),0)$\\
\hline
\end{tabular}
\end{center}

\section{Fibration 3 : $A_6,A_6,A_6$}
Weierstrass equation: $y^2=x^3+(t^4-t+1)x^2+t^2(t-1)(1+t-t^2)x+t^4(t-1)^2$\\
The original model over $\mathbb{Q}$ was given by Tate as an example of a singular $K3$ surface over $\mathbb{Q}$, that is one with rank $20$. The reduction of Tate's equation in characteristic $3$ gives the above equation, and has fiber configuration $$A_6,A_6,A_6$$
Since it also has a non-torsion section, $(0,t^2(t-1))$, thus the rank of the surface over characteristic $3$ is at least $21$. It is known that any such surface is actually of rank $22$, and is hence supersingular. 

Fiber configuration 
$$A_6,A_6,A_6$$ 
\section{Fibration 4 : $A_1,A_7,D_5,D_5$}

To obtain the Weierstrass equation for this fibration, we use a $2$-neighbor step from Fibration $1$. We compute the explicitly the space of sections of the line bundle $\mathcal{O}(F)$, where $F=2O+2a_0+a_1+a_{11}+b_0+c_0$ is the class of the $D_5$ fiber we are considering. The space of the sections is $2$ dimensional and the ratio of two linearly independent global sections will be an elliptic parameter for $X$, for which the class of the fiber will be $F$. Any global section has a pole of order at most $2$ along $O$, the zero section of Fibration $1$. Also it has at most a double pole along $a_0$, which is the identity component of the $t=0$ fiber of Fibration $1$, a simple pole along each of $a_1$ and $a_{11}$ and along $b_0$ and $c_0$, the identity components of the fibers at $t=1$ and $t=\infty$ respectively. We deduce that a global section must have the form 
$$w=\frac{a_0+a_1t+a_2t^2+a_3t^3+a_4t^4+bx}{t^2(t-1)}$$
Since $1$ is a global section, we can subtract multiples of $t^2(t-1)$ from the numerator, and assume that $a_3=0$. With this constraint, we obtain a $1$-dimensional space of sections, the generator of which will give us the parameter for the base of the Fibration $4$. To ensure a unique choice of the generator, we scale $x$ such that $b=1$. \\
The next step is to obtain further conditions on the $a_i$'s by looking at the order of vanishing of the section at various non-identity components (as prescribed above). For this purpose, we employ successive blow-ups at the singular points, and follow Tate's algorithm to obtain conditions. For example, at $t=0$, the first blow-up maps $x\mapsto tx$ and $y\mapsto ty$, which ensures $a_0=0$ for $w$ to have a pole of order at most $1$ at the near leaf of the $A_{11}$ fiber. Similarly, the second blow-up maps $x\mapsto t^2x$, which means $a_1=0$. At $t=\infty$, we change coordinates,and define $u=\frac{1}{t}$, such that $u=0$. Thus, the original equation becomes 
$$y^2=x^3+2(u+u^4)x^2+u^2x$$ and the section becomes 
$$w=\frac{a_2u^2+a_4+xu^3}{u(1-u)}$$
Again, blowing-up once gives $a_4=0$. At $t=1$, we again change coordinates, setting $u=t-1$, such that $u=0$. Now, there is one extra step that needs to be taken care of. The singularity of the curve at $t=1$ is at $(1,0)$. Hence, we map $x\mapsto 1+x$ to map shift the singularity to $(0,0)$. Thus, the original equation changes to 
$$y^2=x^3+2((u+1)^3+1)x^2+(u+1)^6x$$ and the section changes to 
$$w=\frac{x+a_2t^2}{t^2(t-1)}=\frac{x+a_2(u+1)^2}{(u+1)^2u}$$
The first blow-up maps $x\mapsto 1+ux$, giving $a_2=-1$. Thus $w=\frac{x-t^2}{t^2(t-1)}$, which means
$$x=wt^2(t-1)+t^2$$

We now replace $x$ by the above in the right side of the original equation, and simplify by absorbing square factors from the right into the $y^2$ term on the left. We get a quartic in $t$, with coefficients in $\mathbb{F}_3(w)$, which can be converted into the following Weierstrass form, using the recipe described in the Appendix. 

$$y^2=x^3+(t^3+t-1)x^2-t(t-1)(t+1)^2x+t^2(t-1)^2(t+1)^4$$

\begin{center}
\begin{tabular}{|l|c|}\hline
Position & Kodaira-N$\acute{\text{e}}$ron type\\
\hline \hline
$t=0$ & $A_{1}$\\
\hline
$t=1$ & $A_7$\\
\hline
$t=-1$ & $D_5$\\
\hline
$t=\infty$ & $D_5$\\
\hline
\end{tabular}
\end{center} 

This fibration has $MW$-rank of $2$. 

\begin{center}
\begin{tabular}{|l|c|}\hline
section type & equation\\
\hline \hline
non-torsion & $(t(t-1)(t+1)^2,0)$\\
\hline
\end{tabular}
\end{center}

\section{Fibration 5: $E_7,A_2,D_{10}$}
We use a $2$-neighbor step from Fibration $1$, and explicitly compute the $2$-dimensional space of sections of $\mathcal{O}(F)$, where $F = a_1+a_2+2a_0+2O+2c_0+2c_2+2c_3+2c_4+2c_5+c_6+c_7$ is the new $D_{10}$ fiber considered. The new section is $$w=\frac{a_0+a_1t+a_3t^3+a_4t^4+x}{t^2}$$
By blowing up at $t=0$, we get $a_0=a_1=0$. Similarly, subsequent blow-ups at $t=\infty$ give $a_3=a_4=0$. Thus, $$w=\frac{x}{t^2}$$ Substituting $x=wt^2$ in the original equation, dividing by suitable factors and converting the resulting quartic into the Weierstrass form, we get 
$$y^2=x^3-t^3x^2+t^3x$$

\begin{center}
\begin{tabular}{|l|c|}\hline
Position & Kodaira-N$\acute{\text{e}}$ron type\\
\hline \hline
$t=0$ & $E_{7}$\\
\hline
$t=1$ & $A_2$\\
\hline
$t=-1$ & $D_{10}$\\
\hline
\end{tabular}
\end{center} 

The trivial lattice is of rank $2+7+2+10=21$. Thus, the $MW$ rank is $1$. We also have a $2$-torsion section given by $(0,0)$. The only non-torsion section is given by $(1,1)$. The height pairing formula requires that the height of the non-torsion section is $5/2$, and it intersects the $A_2$ fiber at the identity component, and the $D_{10}$ fiber at the near leaf.

\begin{center}
\begin{tabular}{|l|c|}\hline
section type & equation\\
\hline \hline
$2$-torsion & $(0,0)$\\
\hline
non-torsion & $(1,1)$\\
\hline
\end{tabular}
\end{center}

\section{Fibration 6: $D_6,D_6,D_6$}

We use a $2$-neighbor construction from Fibration $5$ and explicitly compute the $2$-dimensional space of sections of $\mathcal{O}(F)$. Here $F=a_1 +a_3+2a_0+2O+b_0+c_0$, where $a_0$ is the identity component of the $D_{10}$ fiber, and $a_1$ the near leaf, $b_0$ is the identity component of the $A_2$ fiber, and $c_0$ that of the $E_7$ fiber. We reparametrized the original equation to get the $D_{10}$ fiber at $t=0$, the $A_2$ at $t=1$ and the $E_7$ at $t=\infty$. The original equation now becomes $$y^2=x^3-tx^2+t^5x$$
And the new section is $$w=\frac{a_0+a_1t+a_2t^2+a_4t^4+x}{t^2(t-1)}$$
Consecutive blow-ups at $t=0$ give $a_0=a_1=0$. Similar calculations at $t=\infty$ give $a_4=0$, and similarly the conditions on the degree of poles of the section at $t=1$ demands that $a_2=1$. Thus, $$w=\frac{x+t^2}{t^2(t-1)}$$ which means
$$x=wt^2(t-1)-t^2$$
Replacing $x$ by the above, and simplifying, we get the required Weierstrass equation
$$y^2=x^3+t(1-t^2)x^2-t^2(t^3-1)x$$

\begin{center}
\begin{tabular}{|l|c|}\hline
Position & Kodaira-N$\acute{\text{e}}$ron type\\
\hline \hline
$t=0$ & $D_{6}$\\
\hline
$t=1$ & $D_6$\\
\hline
$t=\infty$ & $D_{6}$\\
\hline
\end{tabular}
\end{center} 

The $MW$-rank of this surface is $2$, and it has full $2$-torsion, as given by the equations below.

\begin{center}
\begin{tabular}{|l|c|}\hline
section type & equation\\
\hline \hline
$2$-torsion & $(0,0)$\\
\hline
 & $(t^3+t^2+t,0)$\\
\hline
 & $(t-t^2,0)$\\
\hline
\end{tabular}
\end{center}

\section{Fibration 7 : $D_8, A_1, A_1, D_8$}
We use a $2$-neighbor step from Fibration $6$. We explicitly compute the $2$-dimensional space of sections of the line bundle $\mathcal{O}(F)$, where $F=a_1+a_3+2a_2+2a_0+2O+2b_0+2b_2+b_1+b_3$ is the new fiber we consider. Here, $a_0$ and $b_0$ are the identity components of the $D_6$ fibers at $t=0$ and $t=1$ respectively. Whereas $a_1$ and $b_1$ are the near components, and $a_2,a_3,b_2,b_3$ are the double components.\\ 
The original equation is given by $$y^2=x^3+t(1-t^2)x^2-t^2(t^3-1)x$$and the new section by $\frac{a_0+a_1t+a_2t^2+a_3t^3+x}{t^2(t+1)^2}$. Blow-ups at $t=0$ force $a_0=a_1=0$. Blow-ups at $t=1$ give $a_2=a_3=-1$. Thus, $$w=\frac{x-t^2-t^3}{t^2(t+1)^2}$$. Substituting $$x=wt^2(t+1)^2+t^2+t^3$$ in the original equation, and simplifying as usual, we get $$y^2=x^3+t(t^2+1)x^2-t^3(t+1)^2x+t^5(t+1)^2$$ This can further be simplified to 
$$y^2=x^3+(t^3+t)x^2+t^4x$$

\begin{center}
\begin{tabular}{|l|c|}\hline
Position & Kodaira-N$\acute{\text{e}}$ron type\\
\hline \hline
$t=0$ & $D_{8}$\\
\hline
$t=1$ & $A_1$\\
\hline
$t=-1$ & $A_1$\\
\hline
$t=\infty$ & $D_{8}$\\
\hline
\end{tabular}
\end{center} 

The $MW$-rank of this surface is $2$, and it has full $2$-torsion, given by the following equations.

\begin{center}
\begin{tabular}{|l|c|}\hline
section type & equation\\
\hline \hline
$2$-torsion & $(0,0)$\\
\hline
 & $(-t,0)$\\
\hline
 & $(-t^3,0)$\\
\hline
\end{tabular}
\end{center}

\section{Fibration 8 : $E_6, D_7, A_5$}
We use a $2$-neighbor construction from Fibration $4$. We compute the space of sections of the line bundle $\mathcal{O}(F)$, where $F=2O+3a_0+2a_1+2a_7+a_2+a_6$ is the new $E_6$ fiber we are considering. Here $a_0$ is the identity component of the $A_7$ fiber at $t=0$, $a_1$ and $a_7$ are the near leaves, and $a_2$ and $a_6$ are the components intersecting the near leaves. The original equation is $$y^2=x^3- (t-1)(t^2+t-1)x^2+(t-1)^2(t+1)^4x+t^2(t+1)^2(t-1)^4$$
The new section is given by $$w=\frac{a_0+a_1t+a_2t^2+a_3t^3+x}{t^3(t-1)}$$
At $t=0$, the singularity is at $x=-1$. Hence, we translate $x \mapsto x-1$ and blow-up to get $a_0=1$. Replacing $x$ by $tx$ in the original equation, we are required to translate $x \mapsto x-t$. The second blow-up maps $x\mapsto t^2x$, and forces $a_1=1$ to give a pole of order $2$. Similarly, after translating $x\mapsto x-t^2$, the third blow-up gives $a_2=1$. Thus, $$w=\frac{1+t+t^2+a_3t^3+x}{t^3(t-1)}$$
At $t=1$, we define $u=t-1$, so that $$w=\frac{1+(u+1)+(u+1)^2+a_3(u+1)^3+x}{u(u+1)^3}$$ Since the singularity is at $x=0$, the first blow-up $x\mapsto ux$, to give $a_3=0$. Thus, $$w=\frac{x+t^2+t+1}{t^3(t-1)}$$ whence $$x=wt^3(t-1)-t^2-t-1$$ Simplifying, as usual, $$y^2=x^3+tx^2-t^3(t+1)^2x+t^5(t+1)^4$$
(It further simplifies to $x^3-tx^2+t^{12}-t^9+t^6$).

\begin{center}
\begin{tabular}{|l|c|}\hline
Position & Kodaira-N$\acute{\text{e}}$ron type\\
\hline \hline
$t=0$ & $D_{7}$\\
\hline
$t=-1$ & $A_5$\\
\hline
$t=\infty$ & $E_{6}$\\
\hline
\end{tabular}
\end{center} 

 This fibration has $MW$-rank $2$. The non-torsion sections are given below.

\begin{center}
\begin{tabular}{|l|c|}\hline
section type & equation\\
\hline \hline
non-torsion & $(0,t^3(t^3+1))$\\
\hline
% & $(0,t^3(t^3+1))$\\
%\hline
\end{tabular}
\end{center}

\section{Fibration 9 : $A_3,A_9,D_6$}
We use a $2$-neighbor construction on Fibration $1$. The original equation is $$y^2=x^3- (t^3+1)x^2+t^6x$$
The new $D_6$ fiber is given by $F= 2O+a_0+b_0+2c_0+2c_1+c_2+c_3$, where the $a_i$'s are the nodes from the $A_{11}$ fiber, the $b_j$'s from the $A_2$ fiber, and $c_k$'s are from $D_7$. The new section is given by $\frac{a_0+a_1t+a_3t^3+a_4t^4+x}{t(t-1)}$ \\
At $t=0$, we blow-up once, sending $x\mapsto tx$, and get $a_0=0$. For $t=\infty$, we replace $u=\frac{1}{t}$, to get $$w=\frac{a_1u^3+a_3u+a_4+x}{u^2(1-u)}$$ The first blow-up, sending $x\mapsto ux$, gives a pole of order $2$ automatically. For the second blow-up, we need to translate $x\mapsto x-u$, to shift the singularity to $(0,0)$. The second blow-up thus sends $x\mapsto u^2x-u$, and we get $a_4=0$. The third blow-up doesn't need a translation, and sends $x\mapsto u^3x-u$, to yield $a_3=1$. A similar calculation at $t=1$ shows that $a_1=a_3$. Thus, $$w=\frac{x+t^3}{t(t-1)}$$ whence $$x=wt(t-1)-t^3-t$$ Substituting the above in the original equation, and simplifying as usual, we get the following new equation
$$y^2=x^3+(t^3-t-1)x^2+t^5x$$

\begin{center}
\begin{tabular}{|l|c|}\hline
Position & Kodaira-N$\acute{\text{e}}$ron type\\
\hline \hline
$t=0$ & $A_9$\\
\hline
$t=1$ & $A_3$\\
\hline
$t=\infty$ & $D_{6}$\\
\hline
\end{tabular}
\end{center} 

This configuration has $MW$-rank $2$, and has a $2$-torsion section given by $(0,0)$

\begin{center}
\begin{tabular}{|l|c|}\hline
section type & equation\\
\hline \hline
$2$-torsion & $(0,0)$\\
\hline
% & $(0,t^3(t^3+1))$\\
%\hline
\end{tabular}
\end{center}

\section{Fibration 10 : $E_7,D_4,E_7$}
We use a $2$-neighbor construction on Fibration $9$. The original equation is $$y^2=x^3+(t^3-t-1)x^2+(t^3-1)(t^2+t)x-t^2(t-1)^4$$
The original fibers are $A_{9}$ at $t=0$ (roots: $a_0,a_1,\cdots$), $A_3$ at $t=1$ (roots: $b_0,b_1,b_2$), and $D_6$ at $t=\infty$ (roots: $c_0,c_1,\cdots$). The new $E_7$ fiber is given by $F=2O+4a_0+3a_1+3a_2+2a_3+2a_4+a_5+a_6$ and the new section by $w=\frac{a_0+a_1t+a_2t^2+a_3t^3+x}{t^4}$ \\

At $t=0$, the singularity is at $(0,0)$. Hence the blow-up maps $x\mapsto tx$, which gives $a_0=0$. For the second blow-up we need to translate $x$ to $x+t$. Thus, the blow up sends $x\mapsto t^2x+t$, giving $a_1=-1$. The third blow-up does not require a translation, and maps $x\mapsto t^3x+t$, which means $a_2=0$. Similarly, we find $a_4=0$. Thus, $$w=\frac{x-t}{t^4}$$ that is, $$x=wt^4+t$$ Substituting the above in the original equation, simplifying as usual, and transforming the resulting quartic into the Weierstrass form, we get the new equation
$$y^2=x^3+t^3(t+1)^2x$$

\begin{center}
\begin{tabular}{|l|c|}\hline
Position & Kodaira-N$\acute{\text{e}}$ron type\\
\hline \hline
$t=0$ & $E_7$\\
\hline
$t=-1$ & $D_4$\\
\hline
$t=\infty$ & $E_7$\\
\hline
\end{tabular}
\end{center} 

This fibration has $MW$-rank of $2$ and a $2$-torsion section

\begin{center}
\begin{tabular}{|l|c|}\hline
section type & equation\\
\hline \hline
$2$-torsion & $(0,0)$\\
\hline
non-torsion & $(t+1,t^3+1)$\\
\hline
& $(t^3(t+1),-t^3(t^3+1))$\\
\hline
\end{tabular}
\end{center}
Note that the second non-torsion section is the sum of the first non-torsion and the torsion section.

\section{Fibration 11 : $A_{15},A_3$}
We use a $2$-neighbor construction on Fibration $7$. The original equation is $$y^2=x^3+(t^3+t)x^2+t^4x$$ The original fibers are $D_{8}$ at $t=0$ (roots: $a_0,a_1,\cdots$), $D_8$ at $t=\infty$ (roots: $b_0,b_1,b_2$), $A_1$ at $t=1$ (roots: $c_0,c_1,\cdots$), and $A_1$ at $t=-1$ (roots: $d_0,d_1,\cdots$). The new $A_{15}$ fiber is given by $F=O+P+a_0+a_1+a_3+a_4+a_5+a_6+a_7+b_0+b_1+b_3+b_4+b_5+b_6+b_7$, where $P$ is the $2$-torsion section $(0,0)$. \\
The new section is $$w=\frac{a_0+a_2t^2}{t}+\frac{y}{xt}$$ The singularity at $t=0$ is at $(0,0)$. The first blow-up sends $x\mapsto tx$ and $y\mapsto ty$, giving a pole of order $1$. Continuing this for both $t=0$ and $t=\infty$, we get the new equation 
$$y^2=x^3+(t^4+1)x^2+(t^4-1)x+t^4-1$$

\begin{center}
\begin{tabular}{|l|c|}\hline
Position & Kodaira-N$\acute{\text{e}}$ron type\\
\hline \hline
$t=0$ & $A_3$\\
\hline
%$t=-1$ & $D_4$\\
%\hline
$t=\infty$ & $A_{15}$\\
\hline
\end{tabular}
\end{center} 

This fibration has $MW$-rank $2$, and a $4$-torsion section described below.

\begin{center}
\begin{tabular}{|l|c|}\hline
section type & equation\\
\hline \hline
$4$-torsion & $(-1,t^2)$\\
\hline
%non-torsion & $(-1,t^2)$\\
%\hline
%& $(-1,t^2)$\\
%\hline
\end{tabular}
\end{center}

\section{Fibration 12 : $A_{11}, E_6$}
We use the $2$-neighbor construction on fibration $11$. The original equation is $$y^2=x^3+(t^4+1)x^2+t^4(1-t^4)x+t^8(1-t^4)$$
The original fibers are $A_{15}$ at $t=0$ (roots: $a_0,a_1,\cdots$), and $A_3$ at $t=\infty$ (roots: $b_0,b_1,b_2,\cdots)$. The new $E_6$ fiber is given by $F=2O+3a_0+2a_1+2a_2+a_3+a_4+b_0$, and the new section $w=\frac{a_0+a_1t+a_2t^2+a_4t^4+x}{t^3}$ \\
The singularity at $t=0$ is $(0,0)$, and none of the first three blow-ups require a translation. Thus we get $a_0=a_1=a_2=0$. At $t=\infty$, we replace $t$ by $\frac{1}{u}$. We need to translate $x\mapsto x-1$. Thus, the first blow-up maps $x\mapsto ux$,a nd implies that $a_4=1$. Thus, $$w=\frac{x+t^4}{t^3}$$ whence $$x=wt^3-t^4$$ Substituting the above in the original equation, simplifying and converting to the Weierstrass form, we get $$y^2=x^3+x^2+t^4x+t^8$$

\begin{center}
\begin{tabular}{|l|c|}\hline
Position & Kodaira-N$\acute{\text{e}}$ron type\\
\hline \hline
$t=0$ & $A_{11}$\\
\hline
%$t=-1$ & $D_4$\\
%\hline
$t=\infty$ & $E_6$\\
\hline
\end{tabular}
\end{center} 

This fibration has $MW$-rank $3$, and a $3$-torsion section, as given below.

\begin{center}
\begin{tabular}{|l|c|}\hline
section type & equation\\
\hline \hline
$3$-torsion & $(0,t^4)$\\
\hline
non-torsion & $(t^2,t^4+t^2)$\\
\hline
& $(-t^2,t^2-t^4)$\\
\hline
& $(t^4,t^6)$\\
\hline
\end{tabular}
\end{center}

\section{Fibration 13 : $D_6,D_{12}$}
We use the $2$-neighbor construction on Fibration $11$. The original equation is $$y^2=x^3+(t^4+1)x^2+t^4(1-t^4)x+t^8(1-t^4)$$
The original fibers are $A_{15}$ at $t=0$ (roots: $a_0,a_1,\cdots$), $A_3$ at $t=\infty$ (roots: $b_0,b_1,b_2,\cdots)$. The new $D_6$ fiber is given by $F=2O+2a_0+a_1+a_2+2b_0+b_1+b_2$, and the new section is $w=\frac{a_0+a_1t+a_3t^3+a_4t^4+x}{t^2}$\\
Like in the previous calculation at $t=0$, $a_0=a_1=0$. At $t=\infty$, replace $t$ by $\frac{1}{u}$ and translate $x\mapsto x-1$. the first blow-up $x\mapsto ux$ gives $a_4=1$, and the second blow-up (no translation needed) $x\mapsto u^2x$ gives $a_3=0$. Thus, $$w=\frac{x+t^4}{t^2}$$ Replacing $$x=wt^2-t^4$$ in the original equation, simplifying and converting to the Weierstrass equation, we get 
$$y^2=x^3-(t^3+t)x^2+t^6x$$

\begin{center}
\begin{tabular}{|l|c|}\hline
Position & Kodaira-N$\acute{\text{e}}$ron type\\
\hline \hline
$t=0$ & $D_{12}$\\
\hline
%$t=-1$ & $D_4$\\
%\hline
$t=\infty$ & $D_6$\\
\hline
\end{tabular}
\end{center} 

This has $MW$-rank of $2$.

\begin{center}
\begin{tabular}{|l|c|}\hline
section type & equation\\
\hline \hline
$2$-torsion & $(0,0)$\\
\hline
%non-torsion & $(t^2,t^4+t^2)$\\
%\hline
%& $(-t^2,t^2-t^4)$\\
%\hline
%& $(t^4,t^6)$\\
%\hline
\end{tabular}
\end{center}

\section{Fibration 14 : $A_6, A_{12}$}
We use the $2$-neighbor construction on Fibration $8$. The original equation is $$y^2=x^3-tx^2-t^4(t+1)x+t^6(t+1)^2$$ 
The original fibers are $D_7$ at $t=0$ (roots: $a_0,a_1,\cdots$), $A_5$ at $t=1$ (roots: $b_0,b_1,b_2,\cdots)$, and $E_6$ at $t=\infty$ (roots: $c_0,c_1,\cdots$). The new $A_{12}$ fiber is given by $O+P+a_0+a_1+a_3+a_4+a_5+a_6+b_0+b_1+b_2+b_3+b_4$, where $P$ is the section $(0,t^3(t+1))$. The new section is $w=\frac{a(t)}{t}+b(t)\frac{y+t^3(t+1)}{xt}$. Since we want a pole of order $1$ at $t=0$ and also at $t=\infty$, we get that deg$(a(t))=2$. Subtracting suitable multiples of $t$ from the fraction, we get $a_1=0$. Also, deg$(b(t))=0$, and by rescaling, we can assume $b=1$. Thus, $$w=\frac{a_0+a_2t^2}{t}+\frac{y+t^3(t+1)}{xt}$$
As usual, blow-ups at $t=0$ and $t=\infty$ gives both $a_0=a_2=0$. Thus, $$w=\frac{y+t^3(t+1)}{tx}$$ that is, $$y=wtx-t^3(t+1)$$ Substituting the above in the original equation, cancelling common factors, converting to the Weierstrass form and simplifying, we get $$y^2=x^3+(t^4-t^3+1)x^2-t(t-1)^2x+t^2$$
\begin{center}
\begin{tabular}{|l|c|}\hline
Position & Kodaira-N$\acute{\text{e}}$ron type\\
\hline \hline
$t=0$ & $A_{6}$\\
\hline
%$t=-1$ & $D_4$\\
%\hline
$t=\infty$ & $A_{12}$\\
\hline
\end{tabular}
\end{center} 

This again has $MW$-rank $2$.

\begin{center}
\begin{tabular}{|l|c|}\hline
section type & equation\\
\hline \hline
%$2$-torsion & $(0,0)$\\
%\hline
non-torsion & $(0,t)$\\
\hline
%& $(-t^2,t^2-t^4)$\\
%\hline
%& $(t^4,t^6)$\\
%\hline
\end{tabular}
\end{center}

\section{Fibration 15 : $A_{11},E_7$}
We use the $2$-neighbor construction on Fibration $14$. The original equation is $$y^2=x^3+(t^4-t+1)x^2-t^5(t-1)^2x+t^{10}$$ 
The original fibers are $A_{12}$ at $t=0$ (roots: $a_0,a_1,\cdots$), and $A_6$ at $t=\infty$ (roots: $b_0,b_1,b_2,\cdots)$. The new $E_{7}$ fiber is given by $2O+4a_0+3a_1+3a_2+2a_3+2a_4+a_5+a_6$. The new section is $w=\frac{a_0+a_1t+a_2t^2+a_3t^3+x}{t^4}$\\
At $t=0$, the singularity is already at $(0,0)$. Hence we do not need to translate, and can blow-up, sending $x\mapsto tx$. This gives $a_0=0$. Similarly, subsequent blow-ups give $a_1=a_2=a_3=0$. Thus, $$w=\frac{x}{t^4}$$ that is, $$x=wt^4$$ Substituting the above in the original equation, dividing by suitable powers of $t$ till we get a quartic in $t$, and then converting into Weierstrass form, we get 
$$y^2=x^3+(1-t)x^2+t^2(1-t)(1+t)^2x+t^4(t^3+t^2+1)$$

\begin{center}
\begin{tabular}{|l|c|}\hline
Position & Kodaira-N$\acute{\text{e}}$ron type\\
\hline \hline
$t=0$ & $A_{11}$\\
\hline
%$t=-1$ & $D_4$\\
%\hline
$t=\infty$ & $E_{7}$\\
\hline
\end{tabular}
\end{center} 

This has $MW$-rank $1$. The non-torsion section is described below.

\begin{center}
\begin{tabular}{|l|c|}\hline
section type & equation\\
\hline \hline
%$2$-torsion & $(0,0)$\\
%\hline
non-torsion & $(t^2,t^3)$\\
\hline
%& $(-t^2,t^2-t^4)$\\
%\hline
%& $(t^4,t^6)$\\
%\hline
\end{tabular}
\end{center}

\section{Fibration 16 : $A_5, A_5, A_5, D_4$}
We use the $2$-neighbor construction on Fibration $4$. The original equation is 
$$y^2=x^3-(t^3+t-1)x^2+t^3(1-t^2)x+t^6(t-1)^2$$ 
The original fibers are $A_{7}$ at $t=0$ (roots: $a_0,a_1,\cdots$), $A_1$ at $t=1$ (roots: $b_0,b_1$), $D_5$ at $t=-1$ (roots: $c_0,c_1,\cdots$), and $D_5$ at $t=\infty$ (roots: $d_0,d_1,\cdots$). The new $D_{4}$ fiber is given by $2O+a_0+b_0+c_0+d_0$. The new section is $w=\frac{a_0+a_1t+a_2t^2+a_4t^4+x}{t(t-1)(t+1)}$\\
At $t=0$, we blow-up once to get $a_0=0$. At $t=1$, we get $a_1+a_2+a_4=0$. Similarly, at $t=-1$, we get $-a_1+a_2+a_4=1$. At $t=\infty$, blowing-up once gives $a_4=0$. Solving all of the above conditions, we get $a_1=1, a_2=-1$. Thus, $$w=\frac{t-t^2+x}{t(t^2-1)}$$ Substituting $$x=wt(t^2-1)+t^2-t$$ in the original equation, we get, after simplifying, 
$$y^2=x^3+x^2-t^2(t^4-1)x+t^4(t-1)^2(t+1)^2$$

\begin{center}
\begin{tabular}{|l|c|}\hline
Position & Kodaira-N$\acute{\text{e}}$ron type\\
\hline \hline
$t=0$ & $A_{5}$\\
\hline
$t=1$ & $A_5$\\
\hline
$t=-1$ & $A_5$\\
\hline
$t=\infty$ & $D_4$\\
\hline
\end{tabular}
\end{center} 

This fibration has $MW$-rank $1$. The non-torsion section is given below.

\begin{center}
\begin{tabular}{|l|c|}\hline
section type & equation\\
\hline \hline
%$2$-torsion & $(0,0)$\\
%\hline
non-torsion & $(0,t^2(t^2-1))$\\
\hline
%& $(-t^2,t^2-t^4)$\\
%\hline
%& $(t^4,t^6)$\\
%\hline
\end{tabular}
\end{center}

\section{Fibration 17 : $A_8, A_2, A_8$}
We use the $2$-neighbor construction on Fibration $16$. The original equation is $$y^2=x^3+x^2-t^2(t^4-1)x+t^4(t-1)^2(t+1)^2$$ 
The original fibers are $A_{5}$ at $t=0$ (roots: $a_0,a_1,\cdots$), $A_5$ at $t=1$ (roots: $b_0,b_1,\cdots$), $A_5$ at $t=-1$ (roots: $c_0,c_1,\cdots$), and $D_4$ at $t=\infty$ (roots: $d_0,d_1,\cdots$). The new $A_{8}$ fiber is given by $O+P+a_0+a_1+\cdots+b_0+b_1+\cdots$, where $P=(0,t^2(t^2-1))$ is a non-zero section. The new section is $w=\frac{a_0+a_1t}{t(t-1)}+\frac{y+t^2(t^2-1)}{xt(t-1)}$. At $t=0$, we choose the branch of $A_5$ fiber such that $\frac{y}{x}=1$. Thus, $a_0=1$. Similarly, at $t=1$, we replace $t$ by $u+1$, and follow the same logic to get $a_1=0$. Thus, $$w=\frac{-1}{t(t-1)}+\frac{y+t^2(t^2-1)}{xt(t-1)}$$ whereby $$y=wtx(t-1)+x-t^2(t^2-1)$$ We substitute the above in the original equation, and follow the method described in the Appendix, followed by renaming of variables, dividing by suitable squares of polynomials in $t$, and then converting to the Weierstrass form to get 
$$y^2=x^3+(t+1)^4x^2-t^4(t+1)^2(t-1)x+t^8(t+1)^2$$

\begin{center}
\begin{tabular}{|l|c|}\hline
Position & Kodaira-N$\acute{\text{e}}$ron type\\
\hline \hline
$t=0$ & $A_{8}$\\
\hline
%$t=1$ & $A_5$\\
%\hline
$t=-1$ & $A_2$\\
\hline
$t=\infty$ & $A_8$\\
\hline
\end{tabular}
\end{center} 

This has $MW$-rank $2$.

\begin{center}
\begin{tabular}{|l|c|}\hline
section type & equation\\
\hline \hline
%$2$-torsion & $(0,0)$\\
%\hline
non-torsion & $(0,t^4(t+1))$\\
\hline
%& $(-t^2,t^2-t^4)$\\
%\hline
%& $(t^4,t^6)$\\
%\hline
\end{tabular}
\end{center}

\section{Fibration 18 : $D_{10},E_8$}
We do a $2$-neighbor construction on Fibration $15$. The original equation is 
 $$y^2=x^3+(1-t)x^2+t^2(1-t)(1+t)^2x+t^4(t^3+t^2+1)$$ 
The original fibers are $A_{11}$ at $t=0$ (roots: $a_0,a_1,\cdots$), $E_7$ at $t=\infty$ (roots: $b_0,b_1,\cdots$). The new $E_{8}$ fiber is given by $F=2O+a_0+3b_0+4b_1+5b_2+6b_3+3b_4+4b_5+2b_6$. The new section is $w=\frac{a_0+a_2t^2+a_3t^3+a_4t^4+x}{t}$\\
By blowing-up once at $t=0$, we get $a_0=0$. At $t=\infty$, we replace $t$ by $\frac{1}{u}$ as usual, and use subsequent blow-ups to get $a_2=-1$, $a_3=0$, and $a_4=0$. Thus, $$w=\frac{-t^2+x}{t}$$ that is, $$x=wt+t^2$$ Substituting the above in the original equation, dividing by suitable powers of $t$, and converting the resulting quartic into its Weierstrass form, we get

$$y^2=x^3-tx^2+t^9$$

\begin{center}
\begin{tabular}{|l|c|}\hline
Position & Kodaira-N$\acute{\text{e}}$ron type\\
\hline \hline
$t=0$ & $D_{10}$\\
\hline
%$t=1$ & $A_5$\\
%\hline
%$t=-1$ & $A_2$\\
%\hline
$t=\infty$ & $E_8$\\
\hline
\end{tabular}
\end{center} 

This fibration has $MW$-rank $2$.

\begin{center}
\begin{tabular}{|l|c|}\hline
section type & equation\\
\hline \hline
%$2$-torsion & $(0,0)$\\
%\hline
non-torsion & $(t^4,t^6)$\\
\hline
%& $(-t^2,t^2-t^4)$\\
%\hline
%& $(t^4,t^6)$\\
%\hline
\end{tabular}
\end{center}

\section{Fibration 19 : $A_9,A_9$}
We use the $2$-neighbor construction on Fibration $5$. The original equation is $$y^2=x^3-t^3x^2+t^3x$$ 
The original fibers are $E_{7}$ at $t=0$ (roots: $a_0,a_1,\cdots$), $A_2$ at $t=1$ (roots: $b_0,b_1,b_2$), and $D_{10}$ at $t=\infty$ (roots: $c_0,c_1,\cdots$). The new $A_{9}$ fiber is given by $F=O+P+a_0+a_1+a_2+\cdots+b_0$, where $P=(0,0)$ is the $2$-torsion section. The new section is $w=\frac{a_0+a_1t}{t(t-1)}+\frac{y}{xt(t-1)}$\\
Following the usual procedure of blowing up at $t=0$ and $t=1$, we get both $a_0=a_1=0$. Thus, $$w=\frac{y}{xt(t-1)}$$ which means $$y=wt(t-1)x$$ Substituting the above in the original equation, dividing by suitable powers of polynomials of $t$, and renaming variables, we get 

$$y^2=x^3+(t^4+1)x^2-t^2(t^2-1)x+t^4$$

\begin{center}
\begin{tabular}{|l|c|}\hline
Position & Kodaira-N$\acute{\text{e}}$ron type\\
\hline \hline
$t=0$ & $A_{9}$\\
\hline
%$t=1$ & $A_5$\\
%\hline
%$t=-1$ & $A_2$\\
%\hline
$t=\infty$ & $A_9$\\
\hline
\end{tabular}
\end{center} 

 This fibration has $MW$-rank $2$. It also has a $5$-torsion section described below.

\begin{center}
\begin{tabular}{|l|c|}\hline
section type & equation\\
\hline \hline
$5$-torsion & $(0,t^2)$\\
\hline
non-torsion & $(-1,it)$ (defined over $\mathbb{F}_9$)\\
\hline
%& $(-t^2,t^2-t^4)$\\
%\hline
%& $(t^4,t^6)$\\
%\hline
\end{tabular}
\end{center}

\section{Fibration 20 : $A_2,A_{17}$}
We use the $2$-neighbor construction on Fibration $5$. The original equation is $$y^2=x^3-t^3x^2+t^3x$$ 
The original fibers are $E_{7}$ at $t=0$ (roots: $a_0,a_1,\cdots$), $A_2$ at $t=1$ (roots: $b_0,b_1,b_2$), and $D_{10}$ at $t=\infty$ (roots: $c_0,c_1,\cdots$). The new $A_{17}$ fiber is given by $O+P+a_0+a_1+\cdots+c_0+c_1+\cdots$, where $P=(0,0)$ is the $2$-torsion section. The new section is $w=\frac{a_0+a_2t^2}{t}+\frac{y}{xt}$. Following the usual procedure of blowing up at $t=0$ and $t=\infty$, we get both $a_0=a_1=0$. Thus, $$w=\frac{y}{xt}$$ which means $$y=wtx$$
Substituting the above in the original equation, completing squares, and dividing by suitable powers of polynomials of $t$, and renaming variables, we get 

$$y^2=x^3+t^4x^2+t^2x+1$$

\begin{center}
\begin{tabular}{|l|c|}\hline
Position & Kodaira-N$\acute{\text{e}}$ron type\\
\hline \hline
$t=0$ & $A_{2}$\\
\hline
%$t=1$ & $A_5$\\
%\hline
%$t=-1$ & $A_2$\\
%\hline
$t=\infty$ & $A_{17}$\\
\hline
\end{tabular}
\end{center}

This fibration has $MW$-rank of $1$. There is a $3$-torsion section described below.

\begin{center}
\begin{tabular}{|l|c|}\hline
section type & equation\\
\hline \hline
$3$-torsion & $(0,1)$\\
\hline
non-torsion & $(t,t^3+1)$\\
\hline
%& $(-t^2,t^2-t^4)$\\
%\hline
%& $(t^4,t^6)$\\
%\hline
\end{tabular}
\end{center}

\section{Fibration 21 : $A_7,A_7,A_1,A_1,A_1,A_1$}
We use the $2$-neighbor construction on the Fibration $2$. The original equation is $$y^2=x^3-t^2(t-1)^2(t+1)^2$$ 
The original fibers are $D_{4}$ at $t=0$ (roots: $a_0,a_1,\cdots$), $D_4$ at $t=1$ (roots: $b_0,b_1,b_2$), $D_{4}$ at $t=-1$ (roots: $c_0,c_1,\cdots$), and $D_4$ at $t=\infty$ (roots: $d_0,d_1,\cdots$). The new $A_{7}$ fiber is given by $O+P+a_0+a_1+a_2+d_0+d_1+d_2$, where $P=(0,0)$ is the $2$-torsion section. The new section is $w=\frac{a_0+a_2t^2}{t}+\frac{y}{xt}$. Thus, $y=(wt-a_0-a_2t^2)x$. Following the usual procedure of blowing up at $t=0$ and $t=\infty$, we get both $a_0=a_1=0$. Thus, $$w=\frac{y}{xt}$$ which means $$y=wtx$$
Substituting the above in the original equation, dividing by suitable powers of polynomials of $t$, renaming variables, and converting into Weierstrass form, we get 

$$y^2=x^3+(t^4+1)x^2-x-(t^4+1)$$

\begin{center}
\begin{tabular}{|l|c|}\hline
Position & Kodaira-N$\acute{\text{e}}$ron type\\
\hline \hline
$t=0$ & $A_{7}$\\
\hline
$t=1$ & $A_1$\\
\hline
$t=-1$ & $A_1$\\
\hline
$t=i$ & $A_{1}$\\
\hline
$t=-i$ & $A_1$\\
\hline
$t=\infty$ & $A_{7}$\\
\hline
\end{tabular}
\end{center}

This has $MW$-rank rank $2$, and full $2$-torsion .

\begin{center}
\begin{tabular}{|l|c|}\hline
section type & equation\\
\hline \hline
%$3$-torsion & $(0,1)$\\
%\hline
$2$-torsion & $(1,0)$\\
\hline
& $(-1,0)$\\
\hline
& $(-t^4-1,0)$\\
\hline
\end{tabular}
\end{center}

\section{Fibration 22 : $A_{11},A_2,A_2,D_4$}
We use the $2$-neighbor construction on Fibration $21$. The original equation is $$y^2=x^3+(t^4+1)x^2+t^4x$$ 
The original fibers are $A_7$ at $t=0$ (roots: $a_0,a_1,\cdots$), $A_7$ at $t=\infty$ (roots: $b_0,b_1,b_2$), $A_1$ at $t=1$ (roots: $c_0,c_1$),$A_1$ at $t=-1$ (roots: $d_0,d_1$), $A_1$ at $t=i$ (roots: $e_0,e_1$), $A_1$ at $t=-i$ (roots: $f_0,f_1$). The new $D_4$ fiber is given by $F=2O+a_0+b_0+c_0+d_0$. The new section is $w=\frac{a_0+a_1t+a_2t^2+a_4t^4+x}{t(t^2-1)}$.\\
Blowing-up at $t=0$, we get $a_0=0$. At $t=1$, we replace $t$ by $u+1$, and shift $x$ to $x-1$ (so that the singularity is at $(0,0)$), and then blow-up, to get $a_1+a_2+a_4=1$. At $t=-1$, we replace $t$ by $u-1$, shift $x$ to $x-1$, and blow-up to get $-a_1+a_2+a_4=1$. At $t=\infty$, blowing-up gives $a_4=0$. Solving the above equations, we get $$a_1=0,a_2=1,a_4=0$$ Thus, $$w=\frac{t^2+x}{t(t^2-1)}$$ Substituting $$x=wt(t^2-1)-t^2$$ in the original equation, dividing by suitable powers of polynomials of $t$, renaming variables, and converting into Weierstrass form, we get 

$$y^2=x^3+x^2-t^6x$$

\begin{center}
\begin{tabular}{|l|c|}\hline
Position & Kodaira-N$\acute{\text{e}}$ron type\\
\hline \hline
$t=0$ & $A_{11}$\\
%\hline
%$t=1$ & $A_1$\\
%\hline
%$t=-1$ & $A_1$\\
\hline
$t=i$ & $A_{2}$\\
\hline
$t=-i$ & $A_2$\\
\hline
$t=\infty$ & $D_{4}$\\
\hline

\end{tabular}
\end{center} 

This has $MW$-rank $1$, and also a $4$-torsion section.

\begin{center}
\begin{tabular}{|l|c|}\hline
section type & equation\\
\hline \hline
$4$-torsion & $(it^3,it^3)$ (defined over $\mathbb{F}_9$)\\
\hline
non-torsion & $(t^3,t^3)$\\
\hline
%& $(-1,0)$\\
%\hline
%& $(-t^4-1,0)$\\
%\hline
\end{tabular}
\end{center} 

\section{Fibration 23 : $D_6,A_3,A_3,D_6$}
We use the $2$-neighbor construction on Fibration $21$. The original equation is $$y^2=x^3+(t^4+1)x^2+t^4x$$ 
The original fibers are $A_7$ at $t=0$ (roots: $a_0,a_1,\cdots$), $A_7$ at $t=\infty$ (roots: $b_0,b_1,b_2$), $A_1$ at $t=1$ (roots: $c_0,c_1$),$A_1$ at $t=-1$ (roots: $d_0,d_1$), $A_1$ at $t=i$ (roots: $e_0,e_1$), and $A_1$ at $t=-i$ (roots: $f_0,f_1$). The new $D_6$ fiber is given by $2O+2a_0+a_1+a_2+2b_0+b_1+b_2$. The new section is $w=\frac{a_0+a_1t+a_3t^3+a_4t^4+x}{t^2}$. Blowing-up at $t=0$, we get $a_0=0=a_1$. At $t=\infty$, blowing-up, after the usual change of variables, gives $a_3=a_4=0$. Thus, $$w=\frac{x}{t^2}$$ Substituting $$x=wt^2$$ in the original equation, dividing by suitable powers of polynomials of $t$, renaming variables, and converting into Weierstrass form, we get 

$$y^2=x^3+t(t^2+1)x^2-t^4x-t^5(t^2+1)$$

\begin{center}
\begin{tabular}{|l|c|}\hline
Position & Kodaira-N$\acute{\text{e}}$ron type\\
\hline \hline
$t=0$ & $D_6$\\
\hline
$t=1$ & $A_3$\\
\hline
$t=-1$ & $A_3$\\
\hline
%$t=i$ & $A_{2}$\\
%\hline
%$t=-i$ & $A_2$\\
%\hline
$t=\infty$ & $D_6$\\
\hline

\end{tabular}
\end{center} 

This has $MW$-rank of $2$, and full $2$-torsion.

\begin{center}
\begin{tabular}{|l|c|}\hline
section type & equation\\
\hline \hline
%$4$-torsion & $(it^3,it^3)$ (defined over $\mathbb{F}_9$)\\
%\hline
$2$-torsion & $(t^2,0)$\\
\hline
& $(-t^2,0)$\\
\hline
& $(-t^3-t,0)$\\
\hline
\end{tabular}
\end{center}

\section{Fibration 24 : $D_8,D_5,D_5$}
We use the $2$-neighbor construction on Fibration $21$. The original equation is $$y^2=x^3+(t^4+1)x^2+t^4x$$ 
The original fibers are $A_7$ at $t=0$ (roots: $a_0,a_1,\cdots$), $A_7$ at $t=\infty$ (roots: $b_0,b_1,b_2$), $A_1$ at $t=1$ (roots: $c_0,c_1$),$A_1$ at $t=-1$ (roots: $d_0,d_1$), $A_1$ at $t=i$ (roots: $e_0,e_1$), and $A_1$ at $t=-i$ (roots: $f_0,f_1$). The new $D_5$ fiber is given by $2O+2a_0+a_1+a_2+c_0+d_0$. The new section is $w=\frac{a_0+a_1t+a_2t^2+a_3t^3+x}{t^2(t^2-1)}$.\\
Like in the previous case, blowing-up at $t=0$ gives $a_0=0=a_1$. Blowing-up at $t=1$ gives $a_2+a_3=1$, and blowing up at $t=-1$ gives $a_2-a_3=1$. Solving which, we get $a_2=1, a_3=0$. Thus, $$w=\frac{t^2+x}{t^2(t^2-1)}$$ Substituting $$x=wt^2(t^2-1)-t^2$$ in the original equation, dividing by suitable powers of polynomials of $t$, renaming variables, and converting into Weierstrass form, we get 

$$y^2=x^3+(t-t^3)x^2-t^4(t+1)x+t^8$$

\begin{center}
\begin{tabular}{|l|c|}\hline
Position & Kodaira-N$\acute{\text{e}}$ron type\\
\hline \hline
$t=0$ & $D_{8}$\\
\hline
%$t=1$ & $A_3$\\
%\hline
$t=-1$ & $D_5$\\
\hline
%$t=i$ & $A_{2}$\\
%\hline
%$t=-i$ & $A_2$\\
%\hline
$t=\infty$ & $D_5$\\
\hline

\end{tabular}
\end{center} 

This has $MW$-rank of $2$.

\begin{center}
\begin{tabular}{|l|c|}\hline
section type & equation\\
\hline \hline
%$4$-torsion & $(it^3,it^3)$ (defined over $\mathbb{F}_9$)\\
%\hline
$2$-torsion & $(t^3,0)$\\
\hline
non-torsion & $(0,t^4)$\\
\hline
& $(t^2-t,-t^4+t^3-t^2)$\\
\hline
\end{tabular}
\end{center}

\section{Fibration 25 : $A_4,D_5,A_1,A_1,A_7$}
We use the $2$-neighbor construction on Fibration $21$. The original equation is $$y^2=x^3+(t^4+1)x^2+t^4x$$ 
The original fibers are $A_7$ at $t=0$ (roots: $a_0,a_1,\cdots$), $A_7$ at $t=\infty$ (roots: $b_0,b_1,b_2$), $A_1$ at $t=1$ (roots: $c_0,c_1$),$A_1$ at $t=-1$ (roots: $d_0,d_1$), $A_1$ at $t=i$ (roots: $e_0,e_1$), and $A_1$ at $t=-i$ (roots: $f_0,f_1$). The new $A_7$ fiber is given by $F=O+P+a_0+\cdots+c_0+c_1$. The new section is $w=\frac{a_0+a_1t}{t(t-1)}+\frac{y}{xt(t-1)}$.Successive blowing-ups at $t=0$ and $t=1$ give $a_0=1$ and $a_1=-1$. Thus, $$w=\frac{1-t}{t(t-1)}+\frac{y}{xt(t-1)}$$ Substituting $$y=(wt(t-1)-1+t)x$$ in the original equation, dividing by suitable powers of polynomials of $t$, renaming variables, and converting into Weierstrass form, we get $$y^2=x^3+(t^4+t^3+1)x^2+(t-1)^4(t+1)(t^2+1)x+(t-1)^4(t+1)^2(t^2+1)^2$$
which simplifies to 
$$y^2=x^3+(t^4+t^3+1)x^2-(t^4+t^3-t^2+t+1)x;$$

\begin{center}
\begin{tabular}{|l|c|}\hline
Position & Kodaira-N$\acute{\text{e}}$ron type\\
\hline \hline
$t=0$ & $D_5$\\
\hline
%$t=1$ & $A_3$\\
%\hline
$t=-1$ & $A_4$\\
\hline
$t=i$ & $A_1$\\
\hline
$t=-i$ & $A_1$\\
\hline
$t=\infty$ & $A_7$\\
\hline

\end{tabular}
\end{center} 

This has $MW$-rank of $2$.

\begin{center}
\begin{tabular}{|l|c|}\hline
section type & equation\\
\hline \hline
%$4$-torsion & $(it^3,it^3)$ (defined over $\mathbb{F}_9$)\\
%\hline
$2$-torsion & $(0,0)$\\
\hline
%non-torsion & $(0,t^4)$\\
%\hline
%& $(t^2-t,-t^4+t^3-t^2)$\\
%\hline
\end{tabular}
\end{center}

\section{Fibration 26 : $A_5,E_7,D_7$}
We use the $2$-neighbor construction on Fibration $23$. The original equation is $$y^2=x^3+t(t^2+1)x^2-t^4x-t^5(t^2+1)$$ 
The original fibers are $D_6$ at $t=0$ (roots: $a_0,a_1,\cdots$), $A_3$ at $t=1$ (roots: $b_0,b_1,b_2,b_3$), $A_3$ at $t=-1$ (roots: $c_0,c_1,c_2,c_3$),$D_6$ at $t=\infty$ (roots: $d_0,d_1,\cdots$). The new $D_7$ fiber is given by $a_1+a_3+2a_2+2a_0+2O+2b_0+b_1+b_2$. The new section is $w=\frac{a_0+a_1t+a_2t^2+a_3t^3+x}{t^2(t-1)^2}$.\\
Successive blowing-ups at $t=0$ give $a_0=a_1=0$. Blowing-up at $t=1$ after the usual change of variables $u=t-1$, and translating $x \mapsto x+1$, gives $a_2+a_3=-1$. Blowing-up again, after translating $x\mapsto x-u$ gives $a_2=-1$. Hence, $a_3=0$. Thus, $$w=\frac{x-t^2}{t^2(t-1)^2}$$
Substituting $$x=wt^2(t-1)^2+t^2$$ in the original equation, dividing by suitable powers of polynomials of $t$, and renaming variables, we get 

$$y^2=x^3-t^3x^2-t^3x+t^6$$

\begin{center}
\begin{tabular}{|l|c|}\hline
Position & Kodaira-N$\acute{\text{e}}$ron type\\
\hline \hline
$t=0$ & $E_7$\\
\hline
$t=1$ & $A_5$\\
\hline
%$t=-1$ & $A_4$\\
%\hline
%$t=i$ & $A_1$\\
%\hline
%$t=-i$ & $A_1$\\
%\hline
$t=\infty$ & $D_7$\\
\hline

\end{tabular}
\end{center}

This has $MW$-rank $1$. We describe the non-torsion section below.

\begin{center}
\begin{tabular}{|l|c|}\hline
section type & equation\\
\hline \hline
%$4$-torsion & $(it^3,it^3)$ (defined over $\mathbb{F}_9$)\\
%\hline
%$2$-torsion & $(0,0)$\\
%\hline
non-torsion & $(0,t^3)$\\
\hline
%& $(t^2-t,-t^4+t^3-t^2)$\\
%\hline
\end{tabular}
\end{center}

\section{Fibration 27 : $A_5,A_5,D_7$}
We use the $2$-neighbor construction on Fibration $26$. The original equation is $$y^2=x^3+t^3x^2+t^3(t^2-1)x+t^5(t-1)^2$$ 
The original fibers are $A_5$ at $t=0$ (roots: $a_0,a_1,\cdots$), $E_7$ at $t=1$ (roots: $b_0,b_1,\cdots$), $D_7$ at $t=\infty$ (roots: $c_0,c_1,\cdots$). The new $D_7$ fiber is given by $2O+2a_0+a_1+a_2+2c_0+c_1+2c_2+c_3$. The new section is $w=\frac{a_0+a_1t+a_3t^3+a_4t^4+x}{t^2}$.\\
Successive blowing-ups at $t=0$ give $a_0=a_1=0$. Similarly, blowing up at $t=\infty$ gives $a_3=a_4=0$. Thus, $$w=\frac{x}{t^2}$$ Substituting $$x=wt^2$$ in the original equation, dividing by suitable powers of polynomials of $t$, renaming variables, we get 

$$y^2=x^3+(t^3+1)x^2-(t^3-1)x$$

\begin{center}
\begin{tabular}{|l|c|}\hline
Position & Kodaira-N$\acute{\text{e}}$ron type\\
\hline \hline
$t=0$ & $A_5$\\
\hline
$t=1$ & $A_5$\\
\hline
%$t=-1$ & $A_4$\\
%\hline
%$t=i$ & $A_1$\\
%\hline
%$t=-i$ & $A_1$\\
%\hline
$t=\infty$ & $D_7$\\
\hline

\end{tabular}
\end{center} 

This has a $MW$-rank of $3$. It also has full $2$ torsion.

\begin{center}
\begin{tabular}{|l|c|}\hline
section type & equation\\
\hline \hline
%$4$-torsion & $(it^3,it^3)$ (defined over $\mathbb{F}_9$)\\
%\hline
$2$-torsion & $(0,0)$\\
\hline
 & $(1,0)$\\
\hline
& $(1-t^3,0)$\\
\hline
\end{tabular}
\end{center}

\section{Fibration 28 : $A_5,A_{14}$}
We use the $2$-neighbor construction on Fibration $22$. The original equation is $$y^2=x^3+x^2-(t^6+t^2)x-t^4(t^4-t^2-1)$$ 
The original fibers are $A_{11}$ at $t=0$ (roots: $a_0,a_1,\cdots$), $D_4$ at $t=\infty$ (roots: $b_0,b_1,b_2,b_3,b_4$), $A_2$ at $t=1$ (roots: $c_0,c_1,c_2$), and $A_2$ at $t=-1$ (roots: $d_0,d_1,d_2$). The new $A_{14}$ fiber is given by $O+P+a_0+\cdots+c_0+c_1$, where $P=(0,t^2(t^2-1))$ is a non-zero section. The new section is $w=\frac{a_0+a_1t+\frac{y+t^2(t^2-1)}{x}}{t(t-1)}$. Successive blowing-ups at $t=0$ and $t=1$ give $a_0=1$ and $a_1=0$. Thus, $$w=\frac{1}{t(t-1)}+\frac{y+t^2(t^2-1)}{xt(t-1)}$$
Substituting $$y=(wt(t-1)-1)x-t^2(t^2-1)$$ in the original equation, dividing by suitable powers of polynomials in $t$, renaming variables, and converting into Weierstrass form, we get 

$$y^2=x^3+(t^3+1)(t-1)x^2-t^6(t-1)^2x-t^6(t^6+1)$$

\begin{center}
\begin{tabular}{|l|c|}\hline
Position & Kodaira-N$\acute{\text{e}}$ron type\\
\hline \hline
$t=0$ & $A_5$\\
\hline
%$t=1$ & $A_5$\\
%\hline
%$t=-1$ & $A_4$\\
%\hline
%$t=i$ & $A_1$\\
%\hline
%$t=-i$ & $A_1$\\
%\hline
$t=\infty$ & $A_{14}$\\
\hline

\end{tabular}
\end{center}

This has $MW$-rank $1$. 

\section{Fibration 29 : $A_5,A_{5},D_{10}$}
We use the $2$-neighbor construction on Fibration $22$. The original equation is $$y^2=x^3+x^2-(t^6+t^2)x-t^4(t^4-t^2-1)$$ 
The original fibers are $A_{11}$ at $t=0$ (roots: $a_0,a_1,\cdots$), $D_4$ at $t=\infty$ (roots: $b_0,b_1,b_2,b_3,b_4$), $A_2$ at $t=1$ (roots: $c_0,c_1,c_2$), and $A_2$ at $t=-1$ (roots: $d_0,d_1,d_2$). The new $A_{5}$ fiber is given by $O+P+c_0+c_1+d_0+d_1$, where $P=(0,t^2(t^2-1))$ is a non-zero section. The new section is $w=\frac{a_0+a_1t}{(t+1)(t-1)}+\frac{y+t^2(t^2-1)}{x(t+1)(t-1)}$.\\
Successive blowing-ups at $t=0$ and $t=1$ give $a_0=1$ and $a_1=0$. Thus, $$w=\frac{1}{(t+1)(t-1)}+\frac{y+t^2(t^2-1)}{x(t+1)(t-1)}$$ Substituting $$y=(w(t+1)(t-1)-1)x-t^2(t^2-1)$$ in the original equation, dividing by suitable powers of polynomials in $t$, renaming variables, and converting into Weierstrass form, we get 

$$y^2=x^3+(t^3+1)(t-1)x^2-t^6(t-1)^2x-t^6(t^6-1)$$

\begin{center}
\begin{tabular}{|l|c|}\hline
Position & Kodaira-N$\acute{\text{e}}$ron type\\
\hline \hline
$t=0$ & $A_5$\\
\hline
$t=1$ & $D_{10}$\\
\hline
%$t=-1$ & $A_4$\\
%\hline
%$t=i$ & $A_1$\\
%\hline
%$t=-i$ & $A_1$\\
%\hline
$t=\infty$ & $A_5$\\
\hline

\end{tabular}
\end{center} 

The trivial lattice for this fibration is of rank $22$, hence the $MW$-rank is $0$. We also have full $2$-torsion (descriptions given below). The absolute discriminant of the trivial lattice is $6\cdot 6\cdot 4=144$. Then since we have full $2$-torsion, we find, from the Shioda-Tate formula that the full lattice is indeed the full $NS(X)$,since it has absolute discriminant $3^2$, and signature $(1,21)$.

\begin{center}
\begin{tabular}{|l|c|}\hline
section type & equation\\
\hline \hline
%$4$-torsion & $(it^3,it^3)$ (defined over $\mathbb{F}_9$)\\
%\hline
$2$-torsion & $(t^4-t^3,0)$\\
\hline
 & $(-t^4+t^3,0)$\\
\hline
& $(-t^4+t^3-t+1,0)$\\
\hline
\end{tabular}
\end{center}

\section{Fibration 30 : $A_3,A_{9},A_{6}$}
We use the $2$-neighbor construction on Fibration $22$. The original equation is 
$$y^2=x^3+x^2-(t^6+t^2)x-t^4(t^4-t^2-1)$$ 
The original fibers are $A_{11}$ at $t=0$ (roots: $a_0,a_1,\cdots$), $D_4$ at $t=\infty$ (roots: $b_0,b_1,b_2,b_3,b_4$), $A_2$ at $t=1$ (roots: $c_0,c_1,c_2$), and $A_2$ at $t=-1$ (roots: $d_0,d_1,d_2$). The new $A_{6}$ fiber is given by $O+P+c_0+c_1+b_0+b_1+b_2$, where $P=(0,t^2(t^2-1))$ is a non-zero section. The new section is $w=\frac{a_0+a_2t^2}{t-1}+\frac{y+t^2(t^2-1)}{x(t-1)}$.\\
Successive blowing-ups at $t=1$ and $t=\infty$ give $a_0=1$ and $a_1=0$. Thus, $$w=\frac{1}{t-1}+\frac{y+t^2(t^2-1)}{x(t-1)}$$ Substituting $$y=(w(t-1)-1)x-t^2(t^2-1)$$ in the original equation, dividing by suitable powers of polynomials in $t$, renaming variables, and converting into Weierstrass form, we get 

$$y^2=x^3+(t^4+t-1)x^2+t^2(t^3+t^2+1)x+t^4(t-1)^2$$

\begin{center}
\begin{tabular}{|l|c|}\hline
Position & Kodaira-N$\acute{\text{e}}$ron type\\
\hline \hline
$t=0$ & $A_3$\\
\hline
$t=1$ & $A_9$\\
\hline
%$t=-1$ & $A_4$\\
%\hline
%$t=i$ & $A_1$\\
%\hline
%$t=-i$ & $A_1$\\
%\hline
$t=\infty$ & $A_6$\\
\hline

\end{tabular}
\end{center} 

This has $MW$-rank $2$.

\begin{center}
\begin{tabular}{|l|c|}\hline
section type & equation\\
\hline \hline
%$4$-torsion & $(it^3,it^3)$ (defined over $\mathbb{F}_9$)\\
%\hline
non-torsion & $(0,t^2(t-1))$\\
\hline
% & $(-t^4+t^3,0)$\\
%\hline
%& $(-t^4+t^3-t+1,0)$\\
%\hline
\end{tabular}
\end{center}

\section{Fibration 31 : $D_6, A_{12}$}
We use the $2$-neighbor construction on Fibration $24$. The original equation is $$y^2=x^3+(t^3-t^2)x^2-t^3(t+1)x+t^4$$ 
The original fibers are $D_{5}$ at $t=0$ (roots: $a_0,a_1,\cdots$), $D_5$ at $t=-1$ (roots: $b_0,b_1,\cdots$), and $D_{8}$ at $t=\infty$ (roots: $c_0,c_1,\cdots$). The new $A_{12}$ fiber is given by $F=O+P+a_0+a_1+\cdots+c_0+c_1+\cdots$, where $P=(0,t^2)$ is a non-zero section. The new section is $w=\frac{a_0+a_2t^2}{t}+\frac{y+t^2}{xt}$.\\
Successive blowing-ups at $t=0$ and $t=\infty$ give $a_0=0$ and $a_2=0$. Thus, $$w=\frac{y+t^2}{tx}$$ Substituting $$y=wtx-t^2$$ in the original equation, dividing by suitable powers of polynomials in $t$, renaming variables, and converting into Weierstrass form, we get 

$$y^2=x^3+(t^4+t^3+t)x^2+(-t^4+t^2+t)x+t^4-t^3-t^2+t-1$$

\begin{center}
\begin{tabular}{|l|c|}\hline
Position & Kodaira-N$\acute{\text{e}}$ron type\\
\hline \hline
$t=0$ & $D_6$\\
\hline
%$t=1$ & $A_9$\\
%\hline
%$t=-1$ & $A_4$\\
%\hline
%$t=i$ & $A_1$\\
%\hline
%$t=-i$ & $A_1$\\
%\hline
$t=\infty$ & $A_{12}$\\
\hline

\end{tabular}
\end{center} 

This has $MW$-rank $2$.

\section{Fibration 32 : $A_5,A_2,A_5,A_5$}
We use the $2$-neighbor construction on Fibration $25$. The original equation is $$y^2=x^3+(t^4-t^3+t)x^2+t^4(t-1)(t^2-t-1)x+t^4(t-1)^2(t^2-t-1)^2$$ 
The original fibers are $D_{5}$ at $t=0$ (roots: $a_0,a_1,\cdots$), $A_4$ at $t=1$ (roots: $b_0,b_1,\cdots$), $A_7$ at $t=\infty$ (roots: $c_0,c_1,\cdots$), $A_1$ at $t=i-1$ (roots: $d_0,d_1$), and $A_1$ at $t=-i-1$ (roots: $e_0,e_1$). The new $A_{5}$ fiber is given by $F=O+P+b_0+b_1+c_0+c_1$, where $P=(0,t^2(t-1)(t^2-t-1))$ is a non-zero section. The new section is $w=\frac{a_0+a_2t^2}{t-1}+\frac{y+t^2(t-1)(t^2-t-1)}{x(t-1)}$.\\
Successive blowing-ups at $t=1$ and $t=\infty$ give $$a_0=0 \text{ and } a_2=-1$$. Thus, $$w=\frac{-t^2}{t-1}+\frac{y+t^2(t-1)(t^2-t-1)}{x(t-1)}$$ Substituting the above in the original equation, dividing by suitable powers of polynomials in $t$, renaming variables, and converting into Weierstrass form, we get 

$$y^2=x^3+(t^4-t^3-t+1)x^2+(t^8+t^7+t^6)x+t^9+t^6$$

\begin{center}
\begin{tabular}{|l|c|}\hline
Position & Kodaira-N$\acute{\text{e}}$ron type\\
\hline \hline
$t=0$ & $A_5$\\
\hline
$t=1$ & $A_5$\\
\hline
$t=-1$ & $A_5$\\
\hline
%$t=i$ & $A_1$\\
%\hline
%$t=-i$ & $A_1$\\
%\hline
$t=\infty$ & $A_2$\\
\hline

\end{tabular}
\end{center} 

This fibration has $MW$-rank $3$.

\begin{center}
\begin{tabular}{|l|c|}\hline
section type & equation\\
\hline \hline
%$4$-torsion & $(it^3,it^3)$ (defined over $\mathbb{F}_9$)\\
%\hline
non-torsion & $(-t-1,it)$\\
\hline
 & $(t^2,(t-1)t^2(t+1)^2)$\\
\hline
%& $(-t^4+t^3-t+1,0)$\\
%\hline
\end{tabular}
\end{center}

\section{Fibration 33 : $A_8,D_4,E_6$}
We use the $2$-neighbor construction on Fibration $29$. The original equation is $$y^2=x^3+(t^4+t^3-t+1)x^2+(t^5-t^3-t^2-t)x+t^6+t^5+t^3+t^2$$ 
The original fibers are $A_9$ at $t=0$ (roots: $a_0,a_1,\cdots$), $A_6$ at $t=1$ (roots: $b_0,b_1,\cdots$), and $A_3$ at $t=\infty$ (roots: $c_0,c_1,c_2,c_3)$. The new $E_{6}$ fiber is given by $F=2O+3a_0+2a_1+2a_2+a_9+a_8+b_0$. The new section is $w=\frac{a_0+a_1t+a_2t^2+a_4t^4+x}{t^3}$. Successive blowing-ups at $t=0$ give $a_0=0, a_1=1, a_2=-1$ and blowing up at $t=1$ gives $a_4=0$. Thus$$w=\frac{t-t^2+x}{t^3}$$
Substituting $$x=wt^3+t^2-t$$ in the original equation, dividing by suitable powers of polynomials in $t$, renaming variables, and converting into Weierstrass form, we get 

$$y^2=x^3+(t+1)x^2-(t^5+t^4)x+t^8$$

\begin{center}
\begin{tabular}{|l|c|}\hline
Position & Kodaira-N$\acute{\text{e}}$ron type\\
\hline \hline
$t=0$ & $A_8$\\
\hline
%$t=1$ & $A_5$\\
%\hline
$t=-1$ & $D_4$\\
\hline
%$t=i$ & $A_1$\\
%\hline
%$t=-i$ & $A_1$\\
%\hline
$t=\infty$ & $E_6$\\
\hline

\end{tabular}
\end{center} 

This fibration has $MW$-rank $2$.

\begin{center}
\begin{tabular}{|l|c|}\hline
section type & equation\\
\hline \hline
%$4$-torsion & $(it^3,it^3)$ (defined over $\mathbb{F}_9$)\\
%\hline
non-torsion & $(0,t^4)$\\
\hline
% & $(t^2,(t-1)t^2(t+1)^2)$\\
%\hline
%& $(-t^4+t^3-t+1,0)$\\
%\hline
\end{tabular}
\end{center}

\section{Fibration 34 : $A_8,A_2,A_2,D_7$}
We use the $2$-neighbor construction on Fibration $31$. The original equation is $$y^2=x^3+(t^4+t^3+t)x^2-(t^5+t^3)x+t^6-t^5$$ 
The original fibers are $D_6$ at $t=0$ (roots: $a_0,a_1,\cdots$), $A_{12}$ at $t=\infty$ (roots: $b_0,b_1,\cdots$). The new $D_{7}$ fiber is given by $2O+2a_0+2a_1+a_2+a_3+2b_0+b_1+b_2$. The new section is $w=\frac{a_0+a_1t+a_3t^3+a_4t^4+x}{t^2}$. Successive blowing-ups at $t=0$ give $a_0=0$, $a_1=0$ and blowing up at $t=\infty$ gives $a_3=a_4=0$. Thus $$w=\frac{x}{t^2}$$ 
Substituting $$x=wt^2$$ in the original equation, dividing by suitable powers of polynomials in $t$, renaming variables, and converting into Weierstrass form, we get 
 
$$y^2=x^3+(t^3+1)x^2+t^3x+t^6$$

\begin{center}
\begin{tabular}{|l|c|}\hline
Position & Kodaira-N$\acute{\text{e}}$ron type\\
\hline \hline
$t=0$ & $A_8$\\
\hline
%$t=1$ & $A_5$\\
%\hline
%$t=-1$ & $D_4$\\
%\hline
$t=-1-i$ & $A_2$\\
\hline
$t=-1+i$ & $A_2$\\
\hline
$t=\infty$ & $D_7$\\
\hline

\end{tabular}
\end{center} 

This has $MW$-rank $1$.

\begin{center}
\begin{tabular}{|l|c|}\hline
section type & equation\\
\hline \hline
%$4$-torsion & $(it^3,it^3)$ (defined over $\mathbb{F}_9$)\\
%\hline
non-torsion & $(-1,t^3)$\\
\hline
% & $(-1,t^3)$\\
%\hline
%& $(-t^4+t^3-t+1,0)$\\
%\hline
\end{tabular}
\end{center}

\section{Fibration 35 : $D_9,D_9$}
We use the $2$-neighbor construction on Fibration $31$. The original equation is 
$$y^2=x^3+(t^4+t^3+t)x^2-(t^5+t^3)x+t^6-t^5$$ 
The original fibers are $D_6$ at $t=0$ (roots: $a_0,a_1,\cdots$), and $A_{12}$ at $t=\infty$ (roots: $b_0,b_1,\cdots$). The new $D_{9}$ fiber is given by $2O+2a_0+2a_2+2a_3+2a_4+a_5+a_6+2b_0+b_1+b_{12}$. The new section is $w=\frac{a_0+a_1t+a_3t^3+a_4t^4+x}{t^2}$.\\
Successive blowing-ups at $t=0$ give $a_0=0, a_1=0$ and blowing up at $t=\infty$ gives $a_3=1,a_4=0$. Thus $$w=\frac{t^3+x}{t^2}$$ Substituting $$x=wt^2-t^3$$ in the original equation, dividing by suitable powers of polynomials in $t$, renaming variables, and converting into Weierstrass form, we get 

$$y^2=x^3+(t^3+t)x^2+(t^6-t^5)x+t^8$$

\begin{center}
\begin{tabular}{|l|c|}\hline
Position & Kodaira-N$\acute{\text{e}}$ron type\\
\hline \hline
$t=0$ & $D_9$\\
\hline
%$t=1$ & $A_5$\\
%\hline
%$t=-1$ & $D_4$\\
%\hline
%$t=-1-i$ & $A_2$\\
%\hline
%$t=-1+i$ & $A_2$\\
%\hline
$t=\infty$ & $D_9$\\
\hline

\end{tabular}
\end{center} 

This has $MW$-rank $2$.

\begin{center}
\begin{tabular}{|l|c|}\hline
section type & equation\\
\hline \hline
%$4$-torsion & $(it^3,it^3)$ (defined over $\mathbb{F}_9$)\\
%\hline
non-torsion & $(0,t^4)$\\
\hline
% & $(-1,t^3)$\\
%\hline
%& $(-t^4+t^3-t+1,0)$\\
%\hline
\end{tabular}
\end{center}

\section{Fibration 36 : $D_{13},E_6$}
We use the $2$-neighbor construction on Fibration $34$. The original equation is $$y^2=x^3+(t^3+t)x^2+(t^6-t^5)x+t^8$$ 
The original fibers are $D_9$ at $t=0$ (roots: $a_0,a_1,\cdots$), and $D_{9}$ at $t=\infty$ (roots: $b_0,b_1,\cdots$). The new $D_{13}$ fiber is given by $2O+2a_0+2a_1+2a_3+\cdots+2b_0+2b_1+b_2+b_3$. The new section is $w=\frac{a_0+a_1t+a_3t^3+a_4t^4+x}{t^2}$.\\
Successive blowing-ups at $t=0$ and $t=\infty$ give $a_0=a_1=a_3=a_4=0$. Thus $$w=\frac{x}{t^2}$$ Substituting $$x=wt^2-t^3$$ in the original equation, dividing by suitable powers of polynomials in $t$, renaming variables, and converting into Weierstrass form, we get 

$$y^2=x^3-tx^2+t^{12}$$

\begin{center}
\begin{tabular}{|l|c|}\hline
Position & Kodaira-N$\acute{\text{e}}$ron type\\
\hline \hline
$t=0$ & $D_{13}$\\
\hline
%$t=1$ & $A_5$\\
%\hline
%$t=-1$ & $D_4$\\
%\hline
%$t=-1-i$ & $A_2$\\
%\hline
%$t=-1+i$ & $A_2$\\
%\hline
$t=\infty$ & $E_6$\\
\hline

\end{tabular}
\end{center} 

This has $MW$-rank of $1$.

\begin{center}
\begin{tabular}{|l|c|}\hline
section type & equation\\
\hline \hline
%$4$-torsion & $(it^3,it^3)$ (defined over $\mathbb{F}_9$)\\
%\hline
non-torsion & $(0,t^6)$\\
\hline
% & $(-1,t^3)$\\
%\hline
%& $(-t^4+t^3-t+1,0)$\\
%\hline
\end{tabular}
\end{center}

\section{Fibration 37 : $D_{18}$}
We use the $2$-neighbor construction on Fibration $35$. The original equation is $$y^2=x^3+(t^3+t)x^2+(t^6-t^5)x+t^8$$ 
The original fibers are $D_9$ at $t=0$ (roots: $a_0,a_1,\cdots$), $D_{9}$ at $t=\infty$ (roots: $b_0,b_1,\cdots$). The new $D_{18}$ fiber is given by $F= 2O+2a_0+2a_1+2a_3+\cdots+2b_0+2b_1+2b_3+\cdots$. The new section is $w=\frac{a_0+a_1t+a_3t^3+a_4t^4+x}{t^2}$.\\
Successive blowing-ups at $t=0$ and $t=\infty$ give $a_0=a_3=a_4=0, a_1=1$. Thus $$w=\frac{x}{t^2}$$ Substituting $$x=wt^2-t^3$$ in the original equation, dividing by suitable powers of polynomials in $t$, renaming variables, and converting into Weierstrass form, we get 

$$y^2=x^3-(t^4+t^3+t)x^2+(t^7+t^6)x+t^{12}-t^{11}$$

\begin{center}
\begin{tabular}{|l|c|}\hline
Position & Kodaira-N$\acute{\text{e}}$ron type\\
\hline \hline
$t=0$ & $D_{18}$\\
\hline

\end{tabular}
\end{center} 

This again has $MW$-rank $2$. 

\section{Fibration 38 : $E_6,E_6,E_8$}
We use the $2$-neighbor construction on Fibration $24$. The original equation is $$y^2=x^3+(-t^3+t)x^2-t^4(t+1)x+t^8$$ 
The original fibers are $D_8$ at $t=0$ (roots: $a_0,a_1,\cdots$), $D_{5}$ at $t=-1$ (roots: $b_0,b_1,\cdots$), and $D_5$ at $t=\infty$ (roots: $c_0,c_1,\cdots$). The new $E_{8}$ fiber is given by $F= 2O+4a_0+3a_1+6a_2+5a_3+4a_4+3a_5+2a_6+a_7$. The new section is $w=\frac{a_0+a_1t+a_2t^2+a_3t^3+x}{t^4}$.\\
Successive blowing-ups at give $a_0=a_1=a_2=a_3=0$. Thus $$w=\frac{x}{t^4}$$ Substituting $$x=wt^4$$ in the original equation, dividing by suitable powers of polynomials in $t$, renaming variables, and converting into Weierstrass form, we get 

$$y^2=x^3+t^3(t+1)^4$$

\begin{center}
\begin{tabular}{|l|c|}\hline
Position & Kodaira-N$\acute{\text{e}}$ron type\\
\hline \hline
$t=\infty$ & $E_{8}$\\
\hline
 & $E_6$\\
\hline
& $E_6$\\
\hline
%$t=-1-i$ & $A_2$\\
%\hline
%$t=-1+i$ & $A_2$\\
%\hline
%$t=\infty$ & $E_6$\\
%\hline

\end{tabular}
\end{center} 

This pseudo-elliptic fibration has $MW$-rank $0$, since the trivial lattice is already of rank $22$. The absolute discriminant of the trivial lattice is already $3^2$. Thus, the trivial lattice is itself the full $NS(X)$, with signature $(1,21)$ and the prescribed absolute discriminant $3^2$.

\section{Fibration 39 : $A_9,D_9$}
We use the $2$-neighbor construction on Fibration $1$. The original equation is $$y^2=x^3-(t^3+1)x^2+t^6x$$ 
The original fibers are $A_{11}$ at $t=0$ (roots: $a_0,a_1,\cdots$), $A_{2}$ at $t=1$ (roots: $b_0,b_1,b_2$), and $D_7$ at $t=\infty$ (roots: $c_0,c_1,\cdots$). The new $D_{9}$ fiber is given by $F= 2O+a_0+b_0+a_2+2c_0+2c_1+2c_3+\cdots$. The new section is $w=\frac{a_0+a_1t+a_3t^3+a_4t^4+x}{t(t-1)}$.\\
Blowing-up at $t=0$ gives $a_0=0$. Blowing-up at $t=1$ requires translating $x \mapsto x+1 $, and gives $a_1+a_3+a_4=-1$. Blowing-up at $t=\infty$ gives $a_3=a_4=0$. Thus, $a_1=-1$. Thus $$w=\frac{x-t}{t(t-1)}$$ Substituting $$x=wt(t-1)+t$$ in the original equation, dividing by suitable powers of polynomials in $t$, renaming variables, and converting into Weierstrass form, we get 

$$y^2=x^3+(t^3-t-1)x^2+(t^5+t^3)x+t^7-t^6$$

\begin{center}
\begin{tabular}{|l|c|}\hline
Position & Kodaira-N$\acute{\text{e}}$ron type\\
\hline \hline
$t=0$ & $A_{9}$\\
\hline
%$t=1$ & $A_5$\\
%\hline
%$t=-1$ & $D_4$\\
%\hline
%$t=-1-i$ & $A_2$\\
%\hline
%$t=-1+i$ & $A_2$\\
%\hline
$t=\infty$ & $D_9$\\
\hline

\end{tabular}
\end{center} 

This has $MW$-rank $2$.

\begin{center}
\begin{tabular}{|l|c|}\hline
section type & equation\\
\hline \hline
%$4$-torsion & $(it^3,it^3)$ (defined over $\mathbb{F}_9$)\\
%\hline
non-torsion & $(t^2,it^2)$ (defined over $\mathbb{F}_9$)\\
\hline
% & $(-1,t^3)$\\
%\hline
%& $(-t^4+t^3-t+1,0)$\\
%\hline
\end{tabular}
\end{center}

\section{Fibration 40 : $A_2,A_2,D_{16}$}
We use the $2$-neighbor construction on Fibration $5$. The original equation is $$y^2=x^3-tx^2+t^5x$$ 
The original fibers are $D_{10}$ at $t=0$ (roots: $a_0,a_1,\cdots$), $A_{2}$ at $t=1$ (roots: $b_0,b_1,b_2$), and $E_7$ at $t=\infty$ (roots: $c_0,c_1,\cdots$). The new $D_{16}$ fiber is given by $F= 2O+2a_0+2a_1+2a_3+\cdots+2c_0+2c_1+2c_3+\cdots$. The new section is $w=\frac{a_0+a_1t+a_3t^3+a_4t^4+x}{t^2}$.\\
Blowing-up at $t=0$ and $t=\infty$ gives $a_0=a_3=a_4=0$, and $a_1=-1$. Thus $$w=\frac{-t+x}{t^2}$$ Substituting $$x=wt^2+t$$ in the original equation, dividing by suitable powers of polynomials in $t$, renaming variables, and converting into Weierstrass form, we get 

$$y^2=x^3+(t^4+t)x^2+t^8x$$

\begin{center}
\begin{tabular}{|l|c|}\hline
Position & Kodaira-N$\acute{\text{e}}$ron type\\
\hline \hline
$t=0$ & $D_{16}$\\
\hline
$t=1$ & $A_2$\\
\hline
%$t=-1$ & $D_4$\\
%\hline
%$t=-1-i$ & $A_2$\\
%\hline
%$t=-1+i$ & $A_2$\\
%\hline
$t=\infty$ & $A_2$\\
\hline

\end{tabular}
\end{center} 

Here the trivial lattice is of rank $22$, hence the $MW$-rank is $0$. The absolute discriminant of the trivial lattice is $3\cdot 3\cdot 4$. And the fibration has a clear $2$-torsion section, $(0,0)$. Thus, the trivial lattice, together with the $2$-torsion section gives a lattice of signature $(1,21)$ and absolute discriminant $3^2$. Hence it must be the full $NS(X)$.

\begin{center}
\begin{tabular}{|l|c|}\hline
section type & equation\\
\hline \hline
%$4$-torsion & $(it^3,it^3)$ (defined over $\mathbb{F}_9$)\\
%\hline
$2$-torsion & $(0,0)$\\
\hline
% & $(-1,t^3)$\\
%\hline
%& $(-t^4+t^3-t+1,0)$\\
%\hline
\end{tabular}
\end{center}

\section{Fibration 41 : $A_2,A_2,A_2,A_2,A_2,A_2,A_2,A_2,A_2,A_2$}
The equation of this surface is due to Ito. 

$$y^2=x^3+t^{10}+t^2$$

The fibers are defined over the base $\mathbb{P}^1(\mathbb{F}_9)$, with an $A_2$ fiber at each point of the projective line. 

\begin{center}
\begin{tabular}{|l|c|}\hline
Position & Kodaira-N$\acute{\text{e}}$ron type\\
\hline \hline
$t=0$ & $A_{2}$\\
\hline
$t=1$ & $A_2$\\
\hline
$t=-1$ & $A_2$\\
\hline
$t=i$ & $A_2$\\
\hline
$t=1+i$ & $A_2$\\
\hline
$t=-1+i$ & $A_2$\\
\hline
$t=-i$ & $A_2$\\
\hline
$t=1-i$ & $A_2$\\
\hline
$t=-1-i$ & $A_2$\\
\hline
$t=\infty$ & $A_2$\\
\hline

\end{tabular}
\end{center} 

For this pseudo-elliptic surface, the trivial lattice is already of rank $22$, and of absolute discriminant $3^{10}$. The torsion group is of order $3^4$, as described below. Together, the trivial lattice and the $3$-torsion sections account for a lattice of signature $(1,21)$ and absolute discriminant $3^2$. Hence this gives the full $NS(X)$.

\begin{center}
\begin{tabular}{|l|c|}\hline
section type & equation\\
\hline \hline
%$4$-torsion & $(it^3,it^3)$ (defined over $\mathbb{F}_9$)\\
%\hline
$3$-torsion & $(t^2,t(t^4-1))$\\
\hline
& $(t^2,-t(t^4-1))$\\
\hline
& $(-t^2,t(t^4+1))$\\
\hline
& $(-t^2,-t(t^4+1))$\\
\hline
& $(t^6,t(t^8-1))$\\
\hline
& $(t^6,-t(t^8-1))$\\
\hline
& $(\frac{1}{t^2},t^5+\frac{1}{t^2})$\\
\hline
& $(\frac{1}{t^2},-t^5-\frac{1}{t^2})$\\
\hline
\end{tabular}
\end{center}

\section{Fibration 42 : $D_4,A_2,A_2,A_5,A_5$}
We use the $2$-neighbor construction on Fibration $40$. The original equation is $$y^2=x^3+t^{10}+t^2$$
The original fibers are $A_2$'s at all the points of $\mathbb{P}^1(\mathbb{F}_9)$. The new $A_5$ fiber we consider is $F=O+P+a_0+a_1+b_0+b_1$, where $a_0$ and $b_0$ are the identity components of the fibers at $t=0$ and $t=1$ respectively, $a_1$ and $b_1$ are non-identity components of the same, and $P$ is the $3$-torsion section $(t^2,-t(t^4-1))$. Our new section is $w=\frac{a_0+a_1t}{t(t-1)}+\frac{y+t(t^4-1)}{(x-t^2)t(t-1)}$.\\
Blowing-up successively at $t=0$ and $t=1$ gives $a_0=a_1=0$. Substituting the same in the above definition of $w$, and replacing $$y=wt(t-1)(x-t^2)-t(t^4-1)$$ in the original equation, converting to the Weierstrass form after completing squares, renaming variables, and then simplifying, we get 
$$y^2=x^3+(t^3+1)x^2+(t^3-t^6)x$$

\begin{center}
\begin{tabular}{|l|c|}\hline
Position & Kodaira-N$\acute{\text{e}}$ron type\\
\hline \hline
$t=0$ & $A_5$\\
\hline
$t=1$ & $A_5$\\
\hline
%$t=-1$ & $D_4$\\
%\hline
$t=-1-i$ & $A_2$\\
\hline
$t=-1+i$ & $A_2$\\
\hline
$t=\infty$ & $D_4$\\
\hline

\end{tabular}
\end{center} 

This has $MW$-rank $2$. This also has a $2$-torsion section, which does not come from a $4$-torsion section.

\begin{center}
\begin{tabular}{|l|c|}\hline
section type & equation\\
\hline \hline
%$4$-torsion & $(it^3,it^3)$ (defined over $\mathbb{F}_9$)\\
%\hline
$2$-torsion & $(0,0)$\\
\hline
non-torsion & $(-1,t^3)$\\
\hline
%& $(-t^4+t^3-t+1,0)$\\
%\hline
\end{tabular}
\end{center}

\section{Fibration 43 : $A_4,A_4,D_5,D_5$}
We use the $2$-neighbor construction on Fibration $41$. The original equation is $$y^2=x^3+(t^3+1)x^2+(t^3-t^6)x$$
The original fibers are $A_5$ at $t=0$ (Roots: $a_0,a_1,\cdots$), $A_5$ at $t=1$ (Roots: $b_0,b_1,\cdots$), $A_2$ at $t=-i-1$, $A_2$ at $t=-i+1$, and $D_4$ at $t=\infty$. The new $D_5$ fiber we consider is $F=O+2a_0+a_1+a_5+b_0$, where $a_0$ and $b_0$ are the identity components of the fibers at $t=0$ and $t=1$ respectively, $a_i$'s are non-identity components of the $t=0$ fiber. Our new section is $w=\frac{a_0+a_1t+a_2t^2+a_4t^4+x}{t^2(t-1)}$.\\
Blowing-up successively at $t=0$ and $t=1$ gives $a_0=a_1=a_2=a_4=0$. Thus, $$w=\frac{x}{t^2(t-1)}$$ Substituting $$x=wt^2(t-1)$$ in the original equation, converting to the Weierstrass form after completing squares, renaming variables, and then simplifying, we get 
$$y^2=x^3-(t^3+t)x^2-t^2(t^2+t-1)x+t^4(t^2+t-1)^2$$

\begin{center}
\begin{tabular}{|l|c|}\hline
Position & Kodaira-N$\acute{\text{e}}$ron type\\
\hline \hline
$t=0$ & $D_5$\\
\hline
%$t=1$ & $A_5$\\
%\hline
%$t=-1$ & $D_4$\\
%\hline
$t=1+i$ & $A_4$\\
\hline
$t=1-i$ & $A_4$\\
\hline
$t=\infty$ & $D_5$\\
\hline

\end{tabular}
\end{center}

This has $MW$-rank $2$.

\begin{center}
\begin{tabular}{|l|c|}\hline
section type & equation\\
\hline \hline
%$4$-torsion & $(it^3,it^3)$ (defined over $\mathbb{F}_9$)\\
%\hline
%$2$-torsion & $(0,0)$\\
%\hline
non-torsion & $(0,t^2(t^2+t-1))$\\
\hline
& $(t^2+t-1,t(t+1)(t^2+t-1))$\\
\hline
\end{tabular}
\end{center}

\section{Fibration 44 : $A_3,A_3,A_3,A_3,A_3,A_3$}
We use the $2$-neighbor construction on Fibration $41$. The original equation is $$y^2=x^3+(t^3+1)x^2+(t^3-t^6)x$$
The original fibers are $A_5$ at $t=0$, $A_5$ at $t=1$, $A_2$ at $t=-i-1$ (Roots: $c_0,c_1,c_2$), $A_2$ at $t=-i+1$ (Roots: $d_0,d_1,d_2$), and $D_4$ at $t=\infty$. The new $A_3$ fiber we consider is $F=O+P+c_0+d_0$, where $c_0$ and $d_0$ are the identity components of the fibers at $t=-i-1$ and $t=i-1$ respectively, and $P$ is the $2$-torsion section $(0,0)$. Our new section is $w=\frac{a_0+a_1t}{t^2-t-1}+\frac{y}{x(t^2-t-1)}$.\\
Blowing-up successively at $t=0$ and $t=1$ gives $a_0=a_1=0$. Thus, $$w=\frac{y}{x(t^2-t-1)}$$ Substituting $$y=wx(t^2-t-1)$$ in the original equation, converting to the Weierstrass form after completing squares, renaming variables, and then simplifying, we get 
$$y^2=x^3-(t^4+1)x^2+x+t^8-t^4$$

\begin{center}
\begin{tabular}{|l|c|}\hline
Position & Kodaira-N$\acute{\text{e}}$ron type\\
\hline \hline
$t=0$ & $A_3$\\
\hline
$t=1$ & $A_3$\\
\hline
$t=-1$ & $A_3$\\
\hline
$t=i$ & $A_3$\\
\hline
$t=-i$ & $A_3$\\
\hline
$t=\infty$ & $A_3$\\
\hline

\end{tabular}
\end{center} 

This has $MW$-rank $2$.

\begin{center}
\begin{tabular}{|l|c|}\hline
section type & equation\\
\hline \hline
%$4$-torsion & $(it^3,it^3)$ (defined over $\mathbb{F}_9$)\\
%\hline
%$2$-torsion & $(0,0)$\\
%\hline
non-torsion & $(t^2,t(t^2-1))$\\
\hline
%& $(-t^4+t^3-t+1,0)$\\
%\hline
\end{tabular}
\end{center}

\section{Fibration 45 : $A_{18}$}
 This was calculated by Schuett in \cite{Sc} 
$$y^2=x^3+(t^4+t^3+1)x^2+(-t^3-t^2+t)x+t^2+t+1$$ 
is the supersingular $K3$ with an $A_{18}$ fiber.
\begin{center}
\begin{tabular}{|l|c|}\hline
Position & Kodaira-N$\acute{\text{e}}$ron type\\
\hline \hline
$t=\infty$ & $A_{18}$\\
\hline

\end{tabular}
\end{center} 

 This has $MW$-rank $2$.

\begin{center}
\begin{tabular}{|l|c|}\hline
section type & equation\\
\hline \hline
%$4$-torsion & $(it^3,it^3)$ (defined over $\mathbb{F}_9$)\\
%\hline
%$2$-torsion & $(0,0)$\\
%\hline
 & $(0,t-1)$\\
\hline
%& $(-t^4+t^3-t+1,0)$\\
%\hline
\end{tabular}
\end{center}

\section{Fibration 46 : $A_5,A_5,A_8$}
We use the $2$-neighbor construction on Fibration $41$. The original equation is $$y^2=x^3+(t^3+1)x^2+(t^3-t^6)x$$
The original fibers are $A_5$ at $t=0$, $A_5$ at $t=1$, $A_2$ at $t=-i-1$ (Roots: $c_0,c_1,c_2$), $A_2$ at $t=-i+1$ (Roots: $d_0,d_1,d_2$), and $D_4$ at $t=\infty$ (Roots: $e_0,e_1,e_2,e_3,e_4$). The new $A_8$ fiber we consider is $F=O+a_0+a_1+a_2+a_3+P+e_0+e_1+e_2$, where $a_0$ and $e_0$ are the identity components of the fibers at $t=0$ and $t=\infty$ respectively, and $P$ is the $2$-torsion section $(0,0)$. Our new section is $w=\frac{a_0+a_1t}{t}+\frac{y}{xt}$. Our new section is $w=\frac{a_0+a_1t}{t}+\frac{y}{xt}$.\\
Blowing-up successively at $t=0$ and $t=\infty$ gives $a_0=1$, $a_1=0$. Thus, $$w=\frac{1}{t}+\frac{y}{xt}$$ Substituting in the original equation, converting to the Weierstrass form after completing squares, renaming variables, and then simplifying, we get 
$$y^2=x^3+(t^4+t)x^2-t^5x+t^6-t^3-1$$

\begin{center}
\begin{tabular}{|l|c|}\hline
Position & Kodaira-N$\acute{\text{e}}$ron type\\
\hline \hline

%$t=i$ & $A_1$\\
%\hline
%$t=-i$ & $A_1$\\
%\hline
%$t=0$ & $A_2$\\
%\hline
$t=i+1$ & $A_5$\\
\hline
$t=-i+1$ & $A_5$\\
\hline
$t=\infty$ & $A_8$\\
\hline

\end{tabular}
\end{center}

This has $MW$-rank $2$.

\section{Fibration 47 : $A_4,A_4,A_4,A_4,A_1,A_1$}
We use the $2$-neighbor construction on Fibration $43$. The original equation is $$y^2=x^3-(t^4+1)x^2+(t^4-1)x+t^8+t^4+1$$
The original fibers are $A_3$'s at $t=0$, $t=1$, $t=-1$, $t=i$, $t=-i$ and at $t=\infty$. The new $A_4$ fiber we consider is $F=O+P+a_0+b_0+b_1$, where $a_0$ and $b_0$ are the identity components of the fibers at $t=0$ and $t=1$ respectively, $b_1$ is a non-identity fiber at $t=1$, and $P$ is the section $(0,t^4-1)$. Our new section is $w=\frac{a_0+a_1t}{t(t-1)}+\frac{y+t^4-1}{xt(t-1)}$.\\
Blowing-up successively at $t=0$ and $t=1$ gives $a_0=a_1=1$. Thus, $$w=\frac{1+t}{t(t-1)}+\frac{y+t^4-1}{xt(t-1)}$$ Substituting in the original equation, converting to the Weierstrass form after completing squares, renaming variables, and then simplifying, we get 
$$y^2=x^3+(-t^4+t^3-t-1)x^2+(t^7-t^6+t^5)x-t^{10}-t^9+t^8-t^7-t^6-t^5$$

\begin{center}
\begin{tabular}{|l|c|}\hline
Position & Kodaira-N$\acute{\text{e}}$ron type\\
\hline \hline

$t=i$ & $A_1$\\
\hline
$t=-i$ & $A_1$\\
\hline
$t=0$ & $A_4$\\
\hline
$t=i+1$ & $A_4$\\
\hline
$t=-i+1$ & $A_4$\\
\hline
$t=\infty$ & $A_4$\\
\hline

\end{tabular}
\end{center} 

This has $MW$-rank $2$.

\section{Fibration 48 : $A_2,A_2,E_8,E_8$}
We use the $3$-neighbor construction on the fiber configuration $E_6,E_6,E_8$. The original equation is $$y^2=x^3+t^4*(t-1)^5$$
The original fibers are $E_6$ at $t=0$ (Roots: $p_0,p_1,p_2, p_3,p_4,p_5,p_6$ where $p_0,p_4,p_6$ are the simple components and $p_2$ is the component with multiplicity $3$), $E_6$ at $t=\infty$ (Roots: $q_0,q_1,q_2,q_3,q_4,q_5,q_6$ where $q_0,q_4,q_6$ are the simple components and $q_2$ is the component with multiplicity $3$), $E_8$ at $t=1$ (Roots: $r_0,r_1,r_2,\dots$). The new $E_8$ fiber we consider is $F=3O+4p_0+5p_1+6p_2+4p_3+2p_4+3p_5+2q_0+q_1$. Our new section is $w=\frac{y+(a_0+a_1t+a_2t^2)x+b_0+b_1t+b_2t^2+b_3t^3+b_5t^5+b_6t^6}{t^4}$. \\
Blowing-up successively at $t=0$ and $t=\infty$ gives $a_0=0,a_1=0,a_2=0,b_0=0,b_1=0,b_2=0,b_3=i,b_5=0,b_6=0$. Thus, $$w=\frac{y+it^3}{t^4}$$ Substituting in the original equation, converting to the Weierstrass form after completing squares, renaming variables, and then simplifying, we get 
$$y^2=x^3+(t^2+1)^2(t^3+t^2+1)$$

This is a pseudo-elliptic fibration with $MW$-rank $0$. Also since the determinant obtained from the fiber components is already $-9$, hence it doesn't have torsion sections either.

\section{Fibration 49 : $A_2,A_2,A_2,A_2,E_6,E_6$}
We use the $3$-neighbor construction on the fiber configuration $E_6,E_6,E_8$. The original equation is $$y^2=x^3+t^4*(t-1)^5$$
The original fibers are $E_6$ at $t=0$ (Roots: $p_0,p_1,p_2, p_3,p_4,p_5,p_6$ where $p_0,p_4,p_6$ are the simple components and $p_2$ is the component with multiplicity $3$), $E_6$ at $t=\infty$ (Roots: $q_0,q_1,q_2,q_3,q_4,q_5,q_6$ where $q_0,q_4,q_6$ are the simple components and $q_2$ is the component with multiplicity $3$), $E_8$ at $t=1$ (Roots: $r_0,r_1,r_2,\dots$, where $r_0$ is the simple component and $r_1$ is the multiplicity $2$ component intersecting $r_0$). The new $E_6$ fiber we consider is $F=3O+2p_0+p_1+2q_0+q_1+2r_0+r_1$. Our new section is $w=\frac{y+(a_0+a_1t+a_2t^2)x+b_0+b_1t+b_2t^2+b_3t^3+b_5t^5+b_6t^6}{t^2(t-1)^2}$. \\
Blowing-up successively at $t=0$ and $t=\infty$ gives $a_0=0,a_1=0,a_2=0,b_0=0,b_1=0,b_2=0,b_3=0,b_5=0,b_6=0$. Thus, $$w=\frac{y}{t^2(t-1)^2}$$ Substituting $y=wt^2(t-1)^2$ in the original equation, converting to the Weierstrass form after completing squares, renaming variables, and then simplifying, we get 
$$y^2=x^3+t^4(t^2+1)^2$$

This is a pseudo-elliptic fibration with $MW$-rank $0$. It has $3$-torsion sections. 

\begin{center}
\begin{tabular}{|l|c|}\hline
section type & equation\\
\hline \hline
$3$-torsion & $(0,t^2(t^2+1))$\\
\hline
 & $(0,-t^2(t^2+1))$\\
\hline
 & $(-t^2,t^2(t^2-1))$\\
\hline
 & $(-t^2,-t^2(t^2-1))$\\
\hline
\end{tabular}
\end{center}

\section{Fibration 50 : $A_2,E_6,E_6,E_6$}
We use a $2$-neighbor construction on the fiber configuration $A_5,D_7,E_6$. The original equation is $$y^2=x^3+(t^4-t)x^2+t^6$$
The original fibers are $A_5$ at $t=0$ (Roots: $p_0,p_1,p_2, p_3,\dots$ where $p_0$ intersects $O$, the zero section), $D_7$ at $t=\infty$ (Roots: $q_0,q_1,q_2,\dots$ where $q_0$ intersects $O$), $E_6$ at $t=1$. The new $E_6$ fiber we consider is $F=2O+3p_0+2p_1+p_2+2p_5+p_4+q_0$. Our new section is $w=\frac{a_0+a_1t+a_2t^2+a_4t^4+x}{t^3}$. \\
Blowing-up successively at $t=0$ and $t=\infty$ gives $a_0=a_1=a_2=a_4=0$. Thus, $$w=\frac{x}{t^3}$$ Substituting $x=wt^3$ in the original equation, converting to the Weierstrass form after completing squares, renaming variables, and then simplifying, we get 
$$y^2=x^3+(t^2+1)t^4(t-1)^2$$

This is another pseudo-elliptic fibration with MW rank $0$. 

\section{Fibration 51 : $A_3,A_3,A_6,A_6$}
We use a $2$-neighbor construction on the fiber configuration $A_8,A_2,A_2,D_7$. The original equation is $$y^2=x^3+(t^3+1)x^2+t^3x+t^6$$
The original fibers are $A_8$ at $t=0$ (Roots: $p_0,p_1,p_2, p_3,\dots$ where $p_0$ intersects $O$, the zero section), $A_2$ at $t=-i-1$ (Roots: $q_0,q_1,q_2$ where $q_0$ intersects $O$), $A_2$ at $t=i-1$ (Roots: $r_0,r_1,r_2$, where $r_0$ intersects $O$), and $D_7$ at $t=\infty$ (Roots:$s_0,s_1,\dots$). The new $A_3$ fiber we consider is $F=O+P+p_0+q_0$, where $P$ is the section $(-1,t^3)$. Our new section is $w=\frac{a_0+a_1t+\frac{y+t^3}{x+1}}{t(t+i+1)}$. \\
Blowing-up successively at $t=0$ and $t=-i-1$ gives $a_0=0,a_1=i-1$. Thus, $$w=\frac{(i-1)t+\frac{y+t^3}{x+1}}{t(t+i+1)}$$ Substituting in the original equation, converting to the Weierstrass form after completing squares, renaming variables, and then simplifying, we get 
$$y^2=x^3+(-it^4-t^3+(i-1)t+1)x^2-t^2(t-i)(t^3+it^2+t-i-1)x+t^4(t-i)^2((1-i)t^2-i)$$

\begin{center}
\begin{tabular}{|l|c|}\hline
Position & Kodaira-N$\acute{\text{e}}$ron type\\
\hline \hline

$t=0$ & $A_6$\\
\hline
$t=1$ & $A_3$\\
\hline
$t=i$ & $A_6$\\
\hline
%$t=i+1$ & $A_5$\\
%\hline
%$t=-i+1$ & $A_5$\\
%\hline
$t=\infty$ & $A_3$\\
\hline

\end{tabular}
\end{center} 

This has MW rank $2$. 

\section{Fibration 52 : $A_6,A_6,D_6$}
We use a $2$-neighbor construction on the fiber configuration $A_3,A_3,A_6,A_6$. The original equation is $$y^2=x^3+(-it^4-(i+1)t^3-t-1)x^2+(-t^6-it^5+t^4-(i+1)t^3+-it^2)x+(1-i)t^8+(1-i)t^7-it^5-(i+1)t^4$$
The original fibers are $A_3$ at $t=0$ (Roots: $p_0,p_1,p_2, p_3$ where $p_0$ intersects $O$, the zero section), $A_3$ at $t=\infty$ (Roots: $q_0,q_1,q_2,q_3$ where $q_0$ intersects $O$), $A_6$ at $t=i$, and $A_6$ at $t=1$. The new $D_6$ fiber we consider is $F=2O+2p_0+p_1+p_3+2q_0+q_1+q_3$. Our new section is $w=\frac{a_0+a_1t+a_3t^3+a_4t^4+x}{t^2}$. \\
Blowing-up successively at $t=0$ and $t=\infty$ gives $a_0=a_1=a_3=a_4=0$. Thus, $$w=\frac{x}{t^2}$$ Substituting $x=wt^2$ in the original equation, converting to the Weierstrass form after completing squares, renaming variables, and then simplifying, we get 
$$y^2=x^3-(it^3+it)*x^2+(t^8-it^7+(i+1)t^6+(-i+1)t^4+it^3-it^2+(i-1)t-1)x+(it^{12}+it^{11}+(i+1)t^{10}-it^9-it^8+t^7-t^6+(i+1)t^5+(i-1)t^4+it^2+i+1)$$

This has MW rank $2$.
\clearpage

\section{Appendix A}

Here we describe the method of $2$-neighbor and $3$-neighbor construction, which I learnt from my advisor Prof. Abhinav Kumar. For a more detailed description, we refer to \cite{Ku1}. 

\textit{2-neighbor and 3-neighbor construction}

$X$ is an elliptic $K3$ surface, defined over the base field $k$ of characteristic $3$, whose minimal model is given by $y^2 = x^3 + a_2(t) x^2 + a_4(t)x + a_6(t)$ ($a_i\in k[t]$, deg $a_i\leq 2i$). $X$ admits an elliptic fibration over $\mathbb{P}^1$, with generic fiber $E/k(t)$. $O$ denotes the zero section and $P,Q$ etc. are other sections. We want to compute the global sections of $\mathcal{O}_X(a_1O+a_2P+a_3Q+\dots+C)$. We know that $\mathcal{O}_X(P)$ is a line bundle of degree $1$, where $P$ is a section. Riemann-Roch implies $h^0(nP)=n$, with sections 
$$1 \text{ generating }H^0(\mathcal{O}_X(P))$$
$$1^2, x  \text{ generating }H^0(\mathcal{O}_X(2P))$$
$$1^3,1.x,y \text{ generating }H^0(\mathcal{O}_X(3P))$$
etc.\\

For the $2$-neighbor construction, we often consider $D = 2O$. The basis for $\Gamma(\mathcal{O}_X(D))$ is in this case given by $1$ and $x$. If $F_1$, an elliptic divisor, is given by $F_1 = D + C$, then we have $2$ linearly independent global sections to $\Gamma(\mathcal{O}_X(F_1))$, given by $1$ and $a(t) +b(t)x$ for some $a(t), b(t)$. The new elliptic parameter, which we denote by $w$ throughout in this paper, is given by the ratio of these two sections. Thus $w=a(t)+b(t)x$. Here, $a(t)$ is a rational function with numerator of degree at most $4$, and $b(t)$ has a constant numerator, which we assume to be $1$ by rescaling $x$. The denominators of $a(t)$ and $b(t)$ are polynomials giving the correct order of poles at the singular fibers. Then we follow Tate's algorithm of successive blow-ups to pin down the coefficients of $a(t)$, $b(t)$ so that they give the correct order of poles at the respective irreducible components of the fibers. Ultimately we write $x = (w - b(t))/a(t)$. This when substituted into the original equation gives 
$$
y^2 = h(t,u)
$$
where $h$, after suitable change of coordinates to absorb square factors into $y^2$, is a cubic or quartic polynomial in $t$.

The other case when $D = O + P$ is similar, and has been explained in \cite{Ku1}.

In some cases we need to use the $3$-neighbor construction, in which our divisor $D$ is of the form $D=3O+G$ or $2O+P+G$ and so on. In this case the basis for the space of global sections is given by $1,x,y$. Thus we look at sections of the form $w=a(t)+b(t)x+c(t)y$, where $a(t)$ is a rational function with numerator a polynomial of degree $6$, the numerator of $b(t)$ is one of degree $2$, and that of $c(t)$ is a constant which we assume to be $1$ by rescaling $y$. All these rational functions have denominators giving the correct order of poles at the singular fibers. We then use successive blow-ups, as usual, to pin down $w$. Next step is to write $y$ in terms of $x$, $t$ and $w$, and replace in the original equation. This we then simplify to get an equation in $x$ and $t$ of total degree $3$, with coefficients in $k[w]$. This we then convert into a Weierstrass form, using a rational point on the curve.  

The method of conversion of the cubic or quartic into the Weierstrass form can be read off from \cite{Ku1}

\section{Acknowledgements}

I thank Abhinav Kumar and Noam Elkies for many helpful discussions and suggestions. The computer algebra systems PARI/gp and Maxima were used in the calculations for this paper. I thank the developers of these programs.


\begin{thebibliography}{Hud}
\bibitem[Ar]{Ar} M.~Artin, \textit{Supersingular K3 surfaces}, Ann. Sci. Ecole Norm. Sup. (4) {\bf 7} (1974), 543-567 (1975). 
\bibitem[Bo]{Bo} R.~Borcherds, \textit{Automorphism groups of Lorentzian lattices\/}, J.~Algebra {\bf 111} (1987). 133--153.
\bibitem[Co]{Co} J.~H.~Conway, \textit{Three lectures on exceptional groups\/}, in \textit{Finite Simple Groups\/}, pp. 215--247, Academic Press, New York, 1971.
\bibitem [I]{I} H.~Ito, \textit{The Mordell-Weil groups of unirational quasi-elliptic surfaces in characteristic 3}, Math. Z. {\bf 211} (1992), 1--40.
\bibitem[Ku1]{Ku1} A.~Kumar, \textit{Elliptic fibrations on a generic Jacobian Kummer surface}, (2011) 	arXiv:1105.1715v2  [math.AG].
\bibitem[Ku2]{Ku2} A.~Kumar, \textit{$K3$ surfaces associated with curves of genus two\/},  Int.\ Math.\ Res.\ Not.\ IMRN \textbf{2008}, no.\ 6, Art.\ ID rnm165, 26 pp.
\bibitem[KS]{KS} M.~Kuwata and T.~Shioda, \textit{Elliptic parameters and defining equations for elliptic fibrations on a Kummer surface\/}, Algebraic geometry in East Asia -- Hanoi 2005, 177--215, Adv.\ Stud.\ Pure Math.\ \textbf{50}, Math.\ Soc.\ Japan, Tokyo, 2008.
\bibitem[N1]{N1} V. V. ~Nikulin, \textit{Finite groups of automorphisms of Kahlerian surfaces of type K3}, Uspehi Mat. Nauk \textbf{31} (1976), no. 2(188), 223-224.
\bibitem[N2]{N2} V. V. ~Nikulin, \textit{Integral symmetric bilinear forms and some of their applications}, Math. USSR Izv. \textbf{14} (1980), 103-167.
\bibitem[Ni]{Ni} K.~Nishiyama, \textit{The Jacobian fibrations on some K3 surfaces and their Mordell-Weil groups\/}, Japan.\ J.\ Math.\ (New Series) \textbf{22} (1996), no.\ 2, 293--347
\bibitem[O]{O} K.~Oguiso, \textit{On Jacobian fibrations on the Kummer surfaces of the product of nonisogenous elliptic curves\/}, J.\ Math.\ Soc.\ Japan \textbf{41} (1989), no.\ 4, 651--680.
\bibitem[Og]{Og} A.~Ogus, \textit{Supersingular K3 crystals}, Asterisque, {\bf 64} (1979), 3--86. 
\bibitem[PSS]{PSS} I.~Piatetski-Shapiro and I.~R.~Shafarevich, \textit{A Torelli theorem for algebraic surfaces of type K3\/}, Math.\ USSR Izv.\ {\bf 5} (1971), 547--587.
\bibitem [Sc]{Sc} M.~Schuett, \textit{K3 surfaces with Picard rank 20}, (2010) 	arXiv:0804.1558v3  [math.NT].
\bibitem[Sh1]{Sh1} T.~Shioda, \textit{Classical Kummer surfaces and Mordell-Weil lattices\/}, Algebraic geometry, 213--221, Contemp.\ Math.\, {\bf 422}, Amer.\ Math.\ Soc., Providence, RI, 2007.
\bibitem[Sh2]{Sh2} T.~Shioda, \textit{On the Mordell-Weil lattices\/}, Comment.\ Math.\ Univ.\ St.\ Paul.\ {\bf 39} (1990), no.\ 2, 211--240.
\bibitem[Si]{Si} J. H.~Silverman, \textit{Advanced topics in the arithmetic of elliptic curves}, Graduate Texts in
Mathematics, 151. New York: Springer-Verlag, (1994).
\bibitem[St]{St} H.~Sterk, \textit{Finiteness results for algebraic $K3$ surfaces\/}, Math.\ Z.\ \textbf{189} (1985), no.\ 4, 507--513.
\bibitem[V]{V} Vinberg, \textit{Some arithmetical discrete groups in Lobachevskii spaces\/}, in Proc.\ Internat.\ Colloq.\ on Discrete Subgroups of Lie Groups and Applications to Moduli (Bombay 1973), pp. 323--348. Oxford University Press, Bombay, 1975.
\end{thebibliography}
\end{document}